\documentclass[11pt,a4paper,reqno, draft]{amsart}
\usepackage{amsmath, amssymb, latexsym, enumerate,  graphicx,tikz, stmaryrd, eufrak,enumitem}

\usepackage{thmtools}

\usepackage[pdftex]{hyperref}
\usepackage{cleveref}
\usepackage{bbm}
\usepackage{cases}
\usepackage{array}
\usepackage{tikz-cd}
\usepackage[all]{xy}

\usepackage[hmarginratio={1:1},vmarginratio={1:1},lmargin=80.0pt,tmargin=90.0pt]{geometry}

\usepackage{tikz}
\tikzset{anchorbase/.style={baseline={([yshift=-0.5ex]current bounding box.center)}}}
\usetikzlibrary{decorations.markings}
\usetikzlibrary{decorations.pathreplacing}
\usetikzlibrary{arrows,shapes,positioning,backgrounds}
\tikzstyle directed=[postaction={decorate,decoration={markings,
    mark=at position #1 with {\arrow{>}}}}]
\tikzstyle rdirected=[postaction={decorate,decoration={markings,
    mark=at position #1 with {\arrow{<}}}}]
    
\usepackage{graphicx}
\usepackage{nicefrac}

 \newlength{\baseunit}               
 \newcount{\numlines}                
\newtheorem{theorem}[subsubsection]{Theorem}
\newtheorem{lemma}[theorem]{Lemma}
\newtheorem{prop}[theorem]{Proposition}
\newtheorem{corollary}[subsubsection]{Corollary}
\newtheorem{conjecture}[theorem]{Conjecture}

\theoremstyle{definition}
\newtheorem{definition}[subsubsection]{Definition}

\newtheorem{remark}[theorem]{Remark}
\newtheorem{notation}[theorem]{Notation}
\newtheorem{example}[subsubsection]{Example}

\newtheorem{question}[theorem]{Question}

\newtheorem{thmA}{Theorem}

\newcommand{\PSh}{\mathrm{PSh}}
\newcommand{\Sh}{\mathrm{Sh}}
\newcommand{\mG}{\mathbb{G}}

\newcommand{\cT}{\mathcal{T}}

\newcommand{\ev}{\mathrm{ev}}
\newcommand{\co}{\mathrm{co}}

\newcommand{\bA}{\mathbf{A}}
\newcommand{\bB}{\mathbf{B}}
\newcommand{\bC}{\mathbf{C}}

\newcommand{\bD}{\mathbf{D}}
\newcommand{\bT}{\mathbf{T}}

\newcommand{\bU}{\mathbf{U}}

\newcommand{\cU}{\mathcal{U}}

\newcommand{\Spec}{\mathrm{Spec}}
\newcommand{\res}{\mathrm{res}}

\newcommand{\charr}{\mathrm{char}}

\newcommand{\cV}{\mathcal{V}}

\newcommand{\Mod}{\mathsf{Mod}}
\newcommand{\Rep}{\mathsf{Rep}}

\newcommand{\Stab}{\mathsf{Stab}}
\newcommand{\Tens}{\mathsf{Tens}}
\newcommand{\PTens}{\mathsf{PTens}}

\newcommand{\sTens}{\mathsf{sTens}}
\newcommand{\sPTens}{\mathsf{sPTens}}

\newcommand{\unit}{{\mathbbm{1}}}
\newcommand{\tto}{\twoheadrightarrow}
\newcommand{\cO}{\mathcal{O}}

\newcommand{\mN}{\mathbb{N}}
\newcommand{\mZ}{\mathbb{Z}}
\newcommand{\mA}{\mathbb{A}}

\newcommand{\mC}{\mathbb{C}}

\newcommand{\End}{\mathrm{End}}

\newcommand{\Ext}{\mathrm{Ext}}
\newcommand{\colim}{\mathrm{colim}}
\newcommand{\Hom}{\mathrm{Hom}}

\newcommand{\Sym}{\mathrm{Sym}}
\newcommand{\id}{\mathrm{id}}

\newcommand{\St}{\mathrm{St}}

\newcommand{\op}{\mathrm{op}}
\newcommand{\Ind}{\mathrm{Ind}}
\newcommand{\Coker}{\mathrm{Coker}}
\newcommand{\coker}{\mathrm{coker}}

\newcommand{\Vecc}{\mathsf{Vec}}

\newcommand{\Tilt}{\mathsf{Tilt}}
\newcommand{\sJ}{\mathsf{J}}

\newcommand{\sMod}{\mathrm{sMod}}
\newcommand{\Tr}{\mathrm{Tr}}
\newcommand{\Aut}{\mathrm{Aut}}
\newcommand{\ad}{\mathrm{ad}}

\begin{document}
\title[Envelopes with quotient property]{Monoidal abelian envelopes with a quotient property}
\author{Kevin Coulembier, Pavel Etingof, Victor Ostrik and Bregje Pauwels}

\date{\today}


\begin{abstract}
We study abelian envelopes for pseudo-tensor categories with the property that every object in the envelope is a quotient of an object in the pseudo-tensor category.
We establish an intrinsic criterion on pseudo-tensor categories for the existence of an abelian envelope satisfying this quotient property. 
This allows us to interpret the extension of scalars and Deligne tensor product of tensor categories as abelian envelopes, and to enlarge the class of tensor categories for which all extensions of scalars and tensor products are known to remain tensor categories. 
For an affine group scheme $G$, we show that pseudo-tensor subcategories of $\Rep G$ have abelian envelopes with the quotient property, and we study many other such examples. 
This leads us to conjecture that all abelian envelopes satisfy the quotient property.
\end{abstract}

\maketitle


\section*{Introduction}

A tensor category over a field $k$ is a $k$-linear abelian rigid monoidal  category in which the endomorphisms of the tensor identity $\unit$ constitute $k$. Tensor categories, especially tensor categories equipped with a symmetric braiding, are modelled on the representation categories $\Rep G$ of affine group schemes $G$. If we relax the condition of being abelian to being pseudo-abelian (additive and idempotent complete) we arrive at what we will call `pseudo-tensor categories'.

A fully faithful embedding of a pseudo-tensor category into a tensor category with an appropriate universal property is known as an abelian envelope of the pseudo-tensor category. Not all pseudo-tensor categories admit an abelian envelope. The question of which pseudo-tensor categories do admit an abelian envelope has attracted a lot of attention recently, see \cite{BE, BEO, CO, AbEnv1, PreTop, CEH, Deligne, EHS} and led to various interesting applications. An empirical observation about the examples appearing in the above papers (as follows for instance from the unifying construction in \cite{AbEnv1}) is that in every case, every object in the tensor category is a quotient of an object in the pseudo-tensor category.  On the other hand, in \cite{PreTop} it was observed that when a pseudo-tensor category $\bD$ can be embedded in a tensor category $\bT$ such that every object in $\bT$ is a quotient of an object in $\bD$, it follows that $\bT$ must be the abelian envelope of $\bD$.

In the current paper we further investigate abelian envelopes, in particular in relation to the above quotient property. Our first main result establishes an intrinsic criterion on a pseudo-tensor category which determines whether there exists an abelian envelope with the quotient property. For brevity, we only state it in the symmetric case here, see Theorem~\ref{ThmThm} for a full statement.

\begin{thmA}
The following are equivalent for a symmetric pseudo-tensor category $\bD$.
\begin{enumerate}
\item $\bD$ admits an abelian envelope with the quotient property.
\item 
\begin{enumerate}
\item For every nonzero morphism $U\to\unit$ in $\bD$, the diagram $U\otimes U\rightrightarrows U\to \unit$ is a coequaliser.
\item For every morphism $f:X\to Y$ in $\bD$ there exists a nonzero $M\in\Ind\bD$ such that the morphism $M\otimes f$ is split in $\Ind\bD$.
\end{enumerate}
\end{enumerate}
\end{thmA}

Our other main results establish new ways to interpret known (constructions with) tensor categories as abelian envelopes.
 The results invariably produce abelian envelopes with the quotient property. Together with the examples referenced above, this leads us to conjecture that all abelian envelopes have the quotient property. Combined with the aforementioned result from \cite{PreTop}, we thus conjecture that an embedding $\bD\subset\bT$ is an abelian envelope if and only if every object in $\bT$ is a quotient of one in $\bD$. That also means that, conjecturally, Theorem~A provides an intrinsic characterisation of which pseudo-tensor categories admit abelian envelopes. We summarise the results (without pursuing complete accuracy) in the following theorem (see~Theorem~\ref{ConjTann},~Proposition~\ref{PropDelTP} and~Proposition~\ref{PropExtSca} respectively).

\begin{thmA}\label{ThmB}\begin{enumerate}
\item If a pseudo-tensor category $\bD$ can be embedded into $\Rep G$ for an affine group scheme $G$, then $\bD$ admits an abelian envelope with the quotient property of the form $\Rep H$ for an affine group scheme $H$.
\item If for two tensor categories $\bT_1,\bT_2$ their Deligne product $\bT_1\boxtimes\bT_2$ is a tensor category (as is known to be the case when $k$ is perfect by~\cite{Del90, EO}), then the latter is the abelian envelope of the ordinary tensor product $\bT_1\otimes_k\bT_2$ and has the quotient property.
\item  If for a  tensor category $\bT$ over $k$ and a field extension $K/k$, the extension of scalars $\bT_K$ is a tensor category over $K$ (as is the case when $k$ is perfect), then the latter is the abelian envelope of the naive extension of scalars $K\otimes_k\bT$ and has the quotient property.
\end{enumerate}
\end{thmA}

We also investigate whether the Deligne product and extension of scalars are always tensor categories, since we know of no counterexamples. In particular, we derive what additional conditions on the envelopes in Theorem~\ref{ThmB} (2) and (3) we need to impose in order for the Deligne product and extension of scalars to be tensor categories. We moreover enlarge the class of tensor categories for which all extensions of scalars are known to be tensor categories. Concretely, we show that any tensor category which admits a tensor functor to a semisimple tensor category has this property, and we generalise the result from \cite{Del14} from tannakian categories to super-tannakian categories. Finally we demonstrate that if the extension of scalars is always a tensor category, then so is Deligne's product of tensor categories.

The paper is organised as follows. In Section~\ref{SecPrel} we recall some standard notions about monoidal categories. In Section~\ref{SecEnv} we discuss the basics of abelian envelopes and provide a self-contained proof for the fact that every tensor category is its own abelian envelope. In Section~\ref{SecCrit} we prove Theorem~A and its generalisation to pseudo-tensor categories without braiding. In Sections~\ref{SecTann}, \ref{SecExt} and \ref{SecDel} we prove the three results in Theorem~B. In Section~\ref{SecEx} we provide further examples and non-examples of abelian envelopes. 

\section{Preliminaries}\label{SecPrel}
We set $\mN=\{0,1,2,\ldots\}$ and let $k$ be a field unless further specified.
\subsection{Tensor categories and pseudo-tensor categories}


By a $k$-linear monoidal category we mean a monoidal category which is $k$-linear and has $k$-linear tensor product.

\subsubsection{}\label{DefPseudo} A $k$-linear monoidal category $(\bD,\otimes,\unit)$ is a {\bf pseudo-tensor category over $k$} if
\begin{enumerate}[label=(\roman*)]
\item $\bD$ is essentially small;
\item $k\to\End(\unit)$ is an isomorphism;
\item every object $X$ in $\bD$ is rigid (has a left dual $X^\ast$ and right dual ${}^\ast X$);
\item $\bD$ is pseudo-abelian (additive and idempotent complete).
\end{enumerate}

Denote by $\PTens$ the 2-category of pseudo-tensor categories over $k$, where 1-morphisms are $k$-linear monoidal functors and 2-morphisms are monoidal natural transformations. The corresponding 2-category of symmetric pseudo-tensor categories is denoted by $\sPTens$. For a monoidal category satisfying (i)-(iii), we can take its pseudo-abelian envelope (the idempotent completion of the additive envelope) to obtain a pseudo-tensor category.

A \textbf{pseudo-tensor subcategory} of $\bD$ is a \emph{full} monoidal subcategory closed under taking duals, direct sums and summands. It is thus again a pseudo-tensor category. 

\begin{example}\label{ExUni}
Let $\bU_0$ denote the $k$-linear rigid monoidal category freely generated by one object. For instance, we can take $\bU_0$ to be the strict monoidal category where the objects form the free monoid on $\{X_i\,|\,i\in\mZ\}$. Then $\bU_0$ can be made into a rigid monoidal category with $X_i^\ast=X_{i+1}$ for all $i\in\mZ$ and where every morphism is a $k$-linear combination of tensor products of (co)evaluations of (products of) the $X_i$ and identity morphisms. We refer to \cite[\S 5]{CSV} for geometric and combinatorial descriptions. Then we denote by $\bU$ the pseudo-abelian envelope of $\bU_0$, which is a pseudo-tensor category. As in \cite[Proposition~10]{CSV}, we find that for any pseudo-tensor category $\bD$ over $k$, evaluation at $X_0\in \bU$ yields an equivalence between $\PTens(\bU,\bD)$ and the groupoid of objects and isomorphisms in $\bD$.
\end{example}

\subsubsection{}\label{trace} Consider a pseudo-tensor category $\bD$ over $k$ and take $X\in\bD$. Recall from \cite[Definition~4.7.1]{EGNO} that for any morphism $a:X\to X^{\ast\ast}$ the left categorical trace $\Tr^L(a)\in k\simeq\End(\unit)$ is the composite
$$\unit\xrightarrow{\co_X} X\otimes X^{\ast}\xrightarrow{a\otimes X^\ast} X^{\ast\ast}\otimes X^\ast\xrightarrow{\ev_{X^\ast}}\unit.$$

\subsubsection{}When the underlying additive category of a pseudo-abelian category is actually abelian, we say it is a {\bf tensor category} over $k$. This notion is more general than in \cite{Del90}, where tensor categories are assumed to be symmetric, and in \cite{EGNO}, where tensor categories are assumed to have all objects of finite length (be artinian). Following \cite{EGNO} we call \emph{exact} $k$-linear (symmetric) monoidal functors between (symmetric) tensor categories \textbf{(symmetric) tensor functors}. We denote  by $\Tens$  the 2-category  of tensor categories over~$k$, tensor functors and monoidal natural transformations. Similarly we have the 2-category $\sTens$ of symmetric tensor categories. 
We denote the symmetric tensor category of finite dimensional vector spaces over $k$ by $\Vecc_k$, and the category of all vector spaces by $\Vecc_k^\infty$.

\subsubsection{} A $k$-linear abelian category is called {\bf artinian (over $k$)} if its objects are of finite length and its morphism spaces finite dimensional over $k$. A tensor category over $k$ is artinian if and only if its objects are of finite length.

\subsubsection{}  We will freely use that a tensor category $\bT$ has the following properties:
	\begin{enumerate}
		\item The unit $\unit$ in a tensor category is simple. This is proved in \cite[Proposition~1.17]{DM} for symmetric tensor categories, but one can rewrite the proof so that it does not use the braiding. A more succinct proof (and without assumption of braiding) for artinian tensor categories can be found in \cite[Theorem~4.3.8]{EGNO}.

		\item The assignments $X\mapsto X^\ast$ and $X\mapsto {}^\ast X$ yield (exact) anti-autoequivalences of~$\bT$.
	\end{enumerate}

\subsubsection{} A morphism in an additive category is {\bf split} if it is a direct sum of an isomorphism and a zero morphism. In particular, a morphism $f:X\to Y$ in an idempotent complete additive category is split if and only if there exists $g:Y\to X$ with $f=fgf$. For example, $f:X\to X$ is split if and only if it is von Neumann regular in $\End(X)$.

\subsection{Biclosed Grothendieck categories}\label{Biclo}

\subsubsection{Biclosed categories}A monoidal category $\bC$ is biclosed if for every $X\in\bC$ the endo\-functors $X\otimes-$ and $-\otimes X$ of $\bC$ have right adjoints $[X,-]_l$ and $[X,-]_r$.
Clearly, any pseudo-tensor category is biclosed, with
$$[X,-]_l={}^\ast X\otimes-\quad\mbox{and}\quad [X,-]_r=-\otimes X^\ast.$$

\subsubsection{}
Now assume that $\bC$ is a monoidal category such that the underlying category is a Grothendieck category. By Freyd’s special adjoint functor theorem, $\bC$ is then biclosed if and only if the tensor product is cocontinuous in each variable. Such categories will be called `biclosed Grothendieck categories'. 
Throughout, we will freely use that every Grothendieck category has enough injective objects.

\begin{lemma}\label{LemNN}
If an object $N$ in a biclosed Grothendieck category $\bC$ is a filtered colimit of rigid objects, then $N\otimes-$ and $-\otimes N$ are exact.
\end{lemma}
\begin{proof}
This follows because the tensor product is cocontinuous and since filtered colimits of short exact sequences are exact.
\end{proof}

\subsubsection{}
For a tensor category $\bT$, the category $\Ind\bT$ is a biclosed Grothendieck category, see for instance~\cite[\S 7]{Del90} or \cite[Theorem~3.5.4]{PreTop}. This biclosed structure is essentially uniquely determined by demanding that $\bT\to\Ind\bT$ be monoidal. 

\begin{lemma}\label{LemIndFaith} Let $\bT$ be a tensor category over $k$.
\begin{enumerate}
\item For objects $M_1,M_2\in \bT$ and $M_3,M_4\in\Ind\bT$, the canonical morphism
$$\Hom(M_1,M_3)\otimes_k\Hom(M_2,M_4)\;\to\; \Hom(M_1\otimes M_2,M_3\otimes M_4)$$
is injective.
\item For $M\in \Ind\bT$, the functors $M\otimes-$ and $-\otimes M$ are exact and, when $M\not=0$, faithful.
\end{enumerate}
\end{lemma}
\begin{proof}
Exactness in part (2) is a special case of Lemma~\ref{LemNN}.

By adjunction, for part (1) we may assume $M_1=\unit=M_2$. Since $\Hom(\unit,M)$ can be interpreted as the maximal subobject of $M\in\Ind\bT$ in the Serre subcategory $\Vecc^\infty\subset\Ind\bT$, we can rewrite the morphism as
$$\Hom(\unit,M_3)\otimes_k\Hom(\unit,M_4)\xrightarrow{\sim}\Hom(\unit,\Hom(\unit, M_3)\otimes M_4)\hookrightarrow\Hom(\unit, M_3\otimes M_4),$$
by left exactness of $\Hom(\unit,-\otimes M_4)$.

Faithfulness in part (2) now also follows. Indeed, by exactness we only need to show that for non-zero $M,N\in\Ind\bT$, we have $M\otimes N\not=0$. Consider non-zero morphisms $X\to M$ and $Y\to N$ for $X,Y\in\bT$. Applying part (1) then yields a non-zero morphism $X\otimes Y\to M\otimes N$.
\end{proof}



\section{Generalities on abelian envelopes}\label{SecEnv}

\subsection{Abelian envelopes}

The following definition comes from \cite{EHS, CEH}, although here we do not assume braidings. We will make the connection with the symmetric case of  \cite{EHS, CEH} in Lemma~\ref{LemSym} below. Denote by $\PTens^{faith}$ the 2-subcategory of $\PTens$ with same objects but only faithful functors.

\begin{definition}\label{DefAbEnv}
For a pseudo-tensor category $\bD$ over $k$, a pair $(F,\bT)$ of a tensor category $\bT$ over $k$ and a fully faithful linear monoidal functor $F:\bD\to\bT$ constitute an {\bf abelian envelope} of $\bD$ if for each tensor category $\bT_1/k$, composition with $F$ induces an equivalence 
$$\Tens(\bT,\bT_1)\;\simeq\; \PTens^{faith}(\bD,\bT_1).$$
\end{definition}
This definition is justified by the many relevant examples referenced in the introduction. However, from a theoretical viewpoint it is logical to generalise the definition as in the following paragraph.

\subsubsection{} Since tensor functors are faithful, see \cite[2.10]{Del90}, the canonical forgetful 2-functor which interprets a tensor category as a pseudo-tensor category yields a 2-functor
$$\Phi\,:\, \Tens\to\PTens^{faith},$$
which we usually omit from notation. This functor does not admit a left adjoint, see Remark~\ref{RemPhi} below, but we can introduce a partially defined adjoint as follows. 

Define $\PTens^{faith}_0$ as the 2-full 2-subcategory of $\PTens^{faith}$ which comprises those pseudo-tensor categories $\bD$ for which there exists a tensor category $\bT$ and a faithful linear monoidal functor $F:\bD\to\bT$ such that for each tensor category $\bT_1/k$, composition with $F$ induces an equivalence 
$$\Tens(\bT,\bT_1)\;\simeq\; \PTens^{faith}(\bD,\bT_1).$$
Such a pair $(F,\bT)$ will be called a {\bf weak abelian envelope} of $\bD$, since it is just the relaxation of Definition~\ref{DefAbEnv} where we do not demand that $F$ be full.


\begin{remark}\label{RemPhi}
\begin{enumerate}
\item By \cite[Mise en garde~5.8]{Deligne} or \cite[Proposition~4.5.2(i)]{PreTop}, or Example~\ref{Klein} below, we know that $\PTens^{faith}_0\not=\PTens^{faith}$.

\item $\Tens$ is actually a 2-full 2-subcategory of $\PTens^{faith}$, in other words $\Phi$ is 2-fully faithful, see Theorem~\ref{Thm1}. This implies also that $\Phi$ actually takes values in $\PTens^{faith}_0$ and that $\PTens^{faith}_0$ is the unique maximal 2-full 2-subcategory on which $\Phi$ admits a left adjoint $\Psi$. The defining functor $F:\bD\to\bT$ above is then precisely evaluation of the unit $\bD\to \Phi\Psi(\bD)$  of the adjunction $\Psi\dashv \Phi$.
\item We will show that $\bD\to \Phi\Psi(\bD)$ is not always full for $\bD\in\PTens^{faith}_0$, see Subsection~\ref{ExBEO}, so there exist weak abelian envelopes which are not abelian envelopes.
\end{enumerate}
\end{remark}

The following recognition result for weak abelian envelopes was obtained in \cite{PreTop}.
\begin{prop}[Corollary~4.4.4 in \cite{PreTop}]\label{PropGF}
Consider a faithful linear monoidal functor $F:\bD\to\bT$ to a tensor category $\bT$ over $k$ such that
\begin{enumerate} 
\item[(G)] Every object in $\bT$ is a quotient of some $F(X)$, with $X\in\bD$.
\item[(F)] For every morphism $a:F(X)\to F(Y)$ in $\bT$ (with $X,Y\in\bD$) there exists a morphism $q:X'\to X$ in $\bD$ such that $F(q)$ is an epimorphism and $a\circ F( q)$ is in the image of~$F$.
\end{enumerate}
Then $(F,\bT)$ is a weak abelian envelope of $\bD$.
\end{prop}

\subsubsection{} As is standard, we say that an object or a collection of objects in a tensor category is {\bf generating} if consecutively taking duals, direct sums, tensor products and subquotients yields all objects. When considering a pseudo-tensor subcategory of a tensor category (which is already closed under the first three of the above operations), this leads to Definition~\ref{DefGen}(1) below. 
Motivated by condition (G) in Proposition~\ref{PropGF}, we also introduce a stronger version of this property.

\begin{definition}\label{DefGen}
Consider a pseudo-tensor subcategory $\bD$ of a tensor category $\bT$.

\begin{enumerate}
\item The subcategory $\bD\subset\bT$ is {\bf generating} if every object in $\bT$ is a {\em subquotient} of an object in $\bD$.
\item The subcategory $\bD\subset\bT$ is {\bf strongly generating} if every object in $\bT$ is a {\em quotient} of an object in $\bD$.

\end{enumerate}
\end{definition}

Note that in Definition~\ref{DefGen}(2), we can replace `quotient' by `subobject', due to the monoidal dualities. We will also call a collection of objects {\bf strongly generating} if consecutively taking duals, direct sums, tensor products and subobjects yields all objects, which is equivalent to saying that consecutively taking duals, direct sum(mand)s, tensor products yields a strongly generating pseudo-tensor subcategory. The two notions in Definition~\ref{DefGen} are not equivalent, see for instance Proposition~\ref{PropExGLV} below.

\begin{example}\label{ExVsg}
Let $k$ be an algebraically closed field and $V$ a finite dimensional vector space. Then $V$ is strongly generating in $\Rep GL(V)$. Indeed, it is a standard consequence of \cite[Proposition~E.6]{Jantzen} that every tilting object in $\Rep GL(V)$ is a quotient of a polynomial in $V,V^\ast$. By \cite[Remark~3.3.2]{CEH}, every representation is a quotient of a tilting module.
\end{example}

\subsection{A conjecture}

It is clear that condition (F) in~Proposition~\ref{PropGF} holds whenever the functor $F$ is full. Using Definition~\ref{DefGen}(2), we can hence restate the following special case:
\begin{theorem}\label{MainThmAbEnv}
A tensor category is the abelian envelope of each of its strongly generating pseudo-tensor subcategories.
\end{theorem}

We conjecture that all abelian envelopes are of the above form.
\begin{conjecture}\label{Conj}
A pair $\bD\subset\bT$ of a pseudo-tensor subcategory in a tensor category is an abelian envelope if and only if it is strongly generating.
\end{conjecture}

Examples of tensor categories $\bT$ for which Conjecture~\ref{Conj} is known to be true are given by tensor categories with enough projective objects, as follows from Example~\ref{ExProj} and Lemma~\ref{LemTGen} below.
We refer to the introduction for a summary of the other evidence pointing towards the conjecture.
\begin{example}\label{ExProj}
Assume that the tensor category $\bT$ has a nonzero projective object. It follows that $\bT$ has enough projective objects and that every projective object is injective by~\cite[Proposition 6.1.3]{EGNO}. Consequently, a pseudo-tensor subcategory $\bD\subset\bT$ is strongly generating if and only if it is generating.
\end{example}

\begin{lemma}\label{LemTGen}
If $\bD\subset\bT$ is an abelian envelope, then $\bD$ is generating in $\bT$.
\end{lemma}
\begin{proof}
Denote by $\bT_1$ the topologising subcategory generated by $\bD$, that is the full subcategory of $\bT$ of all subquotients of objects in $\bD$. It follows easily that $\bT_1$ is also closed under taking duals and tensor products. Hence $\bT_1$ is a tensor subcategory of $\bT$. By the universal property of $\bD\to\bT$ it follows that there exists a tensor functor $\bT\to\bT_1$ such that the composite $\bT\to\bT_1\hookrightarrow \bT$ is the identity. We conclude that $\bT_1\hookrightarrow \bT$ is an equivalence.
\end{proof}

We cannot reverse the implication in Lemma~\ref{LemTGen}; a tensor category is {\em not} the abelian envelope of every generating pseudo-tensor subcategory, see Example~\ref{ExGP}.

\subsection{Morphisms to the tensor unit}
Let $\bD$ denote a pseudo-tensor category. In this subsection, we point to some natural conditions on morphisms to the tensor unit for $\bD$ to have a (weak) abelian envelope. These morphisms also play a central role in the construction of abelian envelopes from~\cite{PreTop} that will be refined in the next section.

\subsubsection{}\label{ExEp}
We denote by $\cU=\cU(\bD)$ the class of all nonzero morphisms $U\to\unit$ in $\bD$. We consider two potential properties of such $u:U\to \unit$:
\begin{enumerate}[label=(\roman*)] 
\item[(Ep)] The morphism $u:U\to\unit$ is an epimorphism in $\bD$.
\item[(Ex)] The diagram $$\xymatrix{U\otimes U\ar@<0.5ex>[r]^-{u\otimes U}\ar@<-0.5ex>[r]_-{U\otimes u}&U\ar[r]^-{u}&\unit}$$
is a coequaliser in $\bD$.
\end{enumerate}
Of course (Ex) implies (Ep) and we will often regard (Ex) via the equivalent formulation that the sequence 
$$\sigma_u:\quad U\otimes U\xrightarrow{u\otimes U-U\otimes u}U\xrightarrow{u}\unit\to 0$$
be exact. Denote the subclasses of morphisms which satisfy (Ex) and (Ep) by $\cU^{ex}\subset\cU^{ep}\subset\cU$. 

An alternative proof of the following proposition is given in \cite[Theorem~4.2.2]{PreTop}.
\begin{prop}\label{PropTEx}
If $\bT$ is a tensor category, then $\cU(\bT)=\cU^{ex}(\bT)$. 
\end{prop}
\begin{proof}
Using the duality $(-)^\ast$ we can instead prove the corresponding claim for nonzero morphisms $\unit\to U$.
Let $u: \unit\to U$ be a nonzero morphism and let $I\in\Ind\bT$ be a (nonzero) injective object. Then the induced monomorphism $I\otimes u:I\to I\otimes U$ between injective objects splits, giving us $f:I\otimes U\to I$ with $f\circ(I\otimes u)=\id_I$. It follows that
$$0\to I\xrightarrow{I\otimes u}I\otimes U\xrightarrow{I\otimes u\otimes U-I\otimes U\otimes u} I\otimes U\otimes U$$
is split exact by the morphism $f\otimes U: I\otimes U\otimes U\to I\otimes U$:
$$(I\otimes u)\circ f+(f\otimes U)\circ (I\otimes u\otimes U-I\otimes U\otimes u)\;=\;f\otimes u+\id_{I\otimes U}-f\otimes u=\id_{I\otimes U}.$$
Since $I\otimes-$ is faithful and exact, the original sequence $\unit\to U\to U\otimes U$ was also exact.
\end{proof}

As an immediate consequence we get some necessary conditions for the existence of linear monoidal functors to tensor categories.
\begin{corollary}\label{CorEmb} Let $\bT$ be a tensor category over some field extension $K/k$.
\begin{enumerate} 
\item If there is a faithful $k$-linear monoidal functor $\bD\to\bT$, then $\cU(\bD)=\cU^{ep}(\bD)$. Moreover, if $k=K$, then 
$$
0\to\bD(\unit,\unit)\to\bD(U,\unit)\to\bD(U\otimes U,\unit)$$
is exact for each $u:U\to\unit$ in $\cU(\bD)$.
\item If there is a fully faithful $k$-linear monoidal functor $\bD\to\bT$, then $\cU(\bD)=\cU^{ex}(\bD)$.
\end{enumerate}
\end{corollary}

We provide some non-examples for properties (Ep) and (Ex). For examples of pseudo-tensor categories where (Ep) is satisfied but (Ex) is not, we refer to \cite[Lemma~2.3.4(i)]{AbEnv1} or \cite[Proposition~4.5.3(i)]{PreTop}. We show that (Ep) is not always satisfied.

\begin{example}\label{ExTria}
Every epimorphism in a triangulated category is split, for instance by~\cite[Corollary 1.2.7.]{Ne}. In particular, if $\bD$ is a triangulated pseudo-tensor category with a faithful $k$-linear monoidal functor to a tensor category, then every nonzero morphism $U\to\unit$ in $\bD$ is a split epimorphism.
\end{example}

A special case of the previous example is the following.

\begin{example}\label{Klein}
Let $k$ be a field of characteristic $p>0$ and consider a finite group $G$ of order divisible by $p$. Let $\bD=\Stab G$ be the quotient of the representation category $\Rep_k G$ with the ideal of projective modules. 
Then $\cU^{ep}(\bD)$ comprises only split morphisms $U\to\unit$. For instance if $G=C_p$ for $p>2$ or $G=C_2\times C_2$ for $p=2$, we have $\cU^{ep}\not=\cU$.
\end{example}

\begin{lemma}\label{LemSur}
Consider an epimorphism $p:V\to U$ and a morphism $u: U\to\unit$ in $\bD$. If $u\circ p$ is in $\cU^{ex}$, then also $u\in\cU^{ex}$.
\end{lemma}
\begin{proof}
This follows easily from the commutative diagram
$$
\xymatrix{U\otimes U\ar@<0.5ex>[r]\ar@<-0.5ex>[r]&U\ar[r]&\unit.\\
V\otimes V\ar@{->>}[u]\ar@<0.5ex>[r]\ar@<-0.5ex>[r]&V\ar@{->>}[u]\ar@{->>}[ur]
}$$
\end{proof}
\begin{question}\label{QueUV}
Consider morphisms  $u:U\to \unit$ and $v:V\to\unit$ in $\bD$. As a special case of Lemma~\ref{LemSur}, we find that if $u\otimes v\in\cU^{ex}$, then $u,v\in\cU^{ex}$. If, on the other hand, both $u,v\in\cU^{ex}$, does it follows that $u\otimes v\in\cU^{ex}$?
\end{question}

\subsection{A tensor category is its own envelope}

The following is a special case of Theorem~\ref{MainThmAbEnv} and provides a converse to the well-known observation that tensor functors are faithful, see \cite[Corollaire~2.10(i)]{Del90}.

\begin{theorem}\label{Thm1}
A linear monoidal functor $F:\bT_1\to\bT_2$ between two tensor categories is faithful if and only if it is exact. In other words, every tensor category is its own abelian envelope.
\end{theorem}
Since this statement is interesting in its own right, we provide an alternative proof which, contrary to the original proof in \cite{PreTop}, can be written entirely in the language of tensor categories. 
\begin{remark} Besides the proof from \cite{PreTop} an the proof below there is a third, even more direct proof for Theorem~\ref{Thm1}, which we briefly sketch. We can observe that a sequence
$$0\to X\to Y\to Z\to 0$$
in $\bT_1$ is exact if and only if the induced sequence
$$0\to X\otimes I\to Y\otimes I\to Z\otimes I\to 0$$
is split exact, for an arbitrary non-zero injective $I\in\Ind\bT_1$. The conclusion then follows since the indisaton of $F$ is again additive, faithful and monoidal. In a more pedestrian approach, one can express this splitting in terms of families of morphisms in $\bT_1$ and use the faithful, additive and monoidal properties of $F$ to derive that it sends exact sequences to exact ones.
\end{remark}

Fix a tensor category $\bT$. For an associative algebra $B$ in $\Ind\bT$ we denote its category of left modules in $\Ind\bT$ by $\Mod_B$. We have the canonical functor $B\otimes-$ from $\bT$ to $\Mod_B$.
\subsubsection{}Consider a morphism $u:\unit\to U$ in $\bT$ and let $A_u$ denote the algebra in $\Ind\bT$ which is the quotient of the tensor algebra $T^\bullet U$ of $U$ with the ideal generated by image of
$$\unit \xrightarrow{(\id_{\unit},-u)}\unit\oplus U \subset T^\bullet U.$$

\begin{lemma}\label{LemAu}Consider a nonzero morphism $u:\unit\to U$ in $\bT$.
\begin{enumerate}
\item As an object in $\Ind\bT$, we can describe $A_u$ as the colimit of the diagram 
$$\xymatrix{
\unit \ar[r]^u &U\ar@<0.5ex>[r]^-{u\otimes U}\ar@<-0.5ex>[r]_-{U\otimes u}&\otimes^2U\ar@<0.7ex>[r]\ar[r]\ar@<-0.7ex>[r]&\otimes^3U\hspace{2mm}\cdots,
}$$
that is the colimit of the functor from the category of finite ordinals and order preserving inclusions induced by $u$. The algebra unit $\eta:\unit\to A_u$ is given by the canonical morphism from $\unit$ to this colimit.
\item For any algebra $B$ in $\Ind\bT$, the morphism $B\otimes u$ is split in $\Mod_B$ if and only if there exists an algebra morphism $A_u\to B$.
\item The algebra $A_u$ is not zero.
\end{enumerate}
\end{lemma}
\begin{proof}
Part (1) is immediate. Part (2) can be proved precisely as in \cite[Exemple~7.12]{Del90}. Indeed, by construction we have a pushout diagram in the category of algebras in $\Ind\bT$
$$\xymatrix{
T^\bullet U\ar[rr]&&A_u\\
T^\bullet \unit\ar[u]\ar[rr]&&\unit\ar[u]
}$$
where the left vertical arrow is induced from $u:\unit\to U\subset T^\bullet U$ and the lower horizontal arrow from $\id_{\unit}:\unit\to\unit$. Now consider a nonzero algebra $B$. Specifying an element in $\Hom_B(B\otimes U, B)$ is the same as specifying an algebra morphism $T^\bullet U\to B$. Demanding that $B\to B\otimes U\to B$ is the identity is the same as demanding that $T^\bullet \unit\to T^\bullet U\to B$ factors via the above morphism $T^\bullet\unit\to\unit$.

For part (3) we consider a nonzero injective object $I$ in $\Ind\bT$. We claim that there exists a morphism $f:I\otimes A_u\to I$ such that the composite $I\stackrel{I\otimes\eta}{\to} I\otimes A_u\stackrel{f}{\to} I$ is the identity, which shows that $A_u$ is nonzero. We can construct $f$ directly as follows. Since $I\otimes u$ is a monomorphism of injective objects there exists a morphism $h:I\otimes U\to I$ with $h\circ(I\otimes u)=\id_I$. We can use this to construct a commutative diagram
$$\xymatrix{
I\ar@{=}[dd] \ar[r]^{I\otimes u} &I\otimes U\ar[ldd]^{h}\ar@<0.5ex>[r]^-{I\otimes u\otimes U}\ar@<-0.5ex>[r]_-{I\otimes U\otimes u}&I\otimes U\otimes U\ar@<0.7ex>[r]\ar[r]\ar@<-0.7ex>[r]
\ar@/^/[ddll]^{h\circ (h\otimes U)}&I\otimes U\otimes U\otimes U
\ar@/^2pc/[ddlll]^[l]{h\circ (h\otimes U)\circ (h\otimes U\otimes U)}\hspace{2mm}\cdots\\
&&\\
I,
}$$
which implies existence of the requested morphism.
Alternatively, we can observe that the internal hom algebra $[I,I]$ is injective and hence splits the morphism $u$ (see~\cite[Section~8]{EO}). The universality in part (2) gives an algebra morphism $A_u\to[I,I]$ showing that $A_u$ is nonzero.
\end{proof}

\begin{lemma}\label{LemSplit}
A sequence $$\xi: 0\to X\xrightarrow{a} Y\xrightarrow{b} Z\to 0$$
 is exact in $\bT$ if and only if there exists a nonzero morphism $u:\unit\to U$ such that the induced chain map $\xi\xrightarrow{u\otimes \xi} U\otimes \xi$ is nullhomotopic.
\end{lemma}
\begin{proof}
Assume first that such $u:\unit\to U$ exists. Since $\eta:\unit\to A_u$ factors via $u$, existence of the chain homotopy implies that also the chain map $\xi\xrightarrow{\eta\otimes \xi} A_u\otimes \xi$ is nullhomotopic, which in turns implies that the identity chain map of $A_u\otimes \xi$ in $\Mod_{A_u}$ is nullhomotopic. The latter just means that $A_u\otimes \xi$ is a split exact sequence in $\Mod_{A_u}$. Since $A_u\otimes-$ is exact and faithful, see Lemma~\ref{LemAu}(3), it follows that $\xi$ is also exact.

Now assume that $\xi$ is exact. First we take the pushout
$$\xymatrix{
X\otimes {}^\ast X\ar@{^{(}->}[rr]^{a\otimes{}^\ast X}\ar@{->>}^{\ev}[d] && Y\otimes {}^\ast X\ar@{->>}[d]\\
\unit\ar@{^{(}->}[rr]^e&& E.
}$$
Applying adjunction shows that the left triangle of the diagram below is commutative
$$
\xymatrix{
X\ar@{^{(}->}[rr]^a\ar@{^{(}->}[d]_{e\otimes X}&&Y\ar@{^{(}->}[d]^{e\otimes Y}\ar[dll]\\
E\otimes X\ar@{^{(}->}[rr]_{E\otimes a}&& E\otimes Y.
}$$
Hence, the difference between the two morphisms $Y\to E\otimes Y$ restricts to zero when composed with $a$. By exactness of $\xi$, this difference factors via $b:Y\tto Z$. This gives the required homotopy.
\end{proof}

Before finishing the proof, we record the following corollary.
\begin{corollary}\label{Corfu}
For any morphism $f:X\to Y$ in $\bT$ there exists a nonzero morphism $u:\unit\to U$ and a morphism $g: Y\to U\otimes X$ with $u\otimes f=(U\otimes f)\circ g\circ f$.
\end{corollary}
\begin{proof}
We can write $f$ as a composition of an epimorphism and monomorphism $X\tto A\hookrightarrow Y$. Applying Lemma~\ref{LemSplit} to the exact sequence $A\hookrightarrow Y\tto ?$ yields morphisms $v:\unit\hookrightarrow V$ and $Y\to V\otimes A$ such that the two paths from $A$ to $V\otimes A$ yield the same morphism in the following diagram:
$$\xymatrix{
X\ar@{->>}[rr]\ar@{^{(}->}[d]&& A\ar@{^{(}->}[rr]\ar@{^{(}->}[d]&& Y\ar@{^{(}->}[d]\ar[dll]\\
V\otimes X\ar@{->>}[rr]\ar@{^{(}->}[d]&& V\otimes A\ar@{^{(}->}[rr]\ar@{^{(}->}[d]\ar[dll]&& V\otimes Y\ar@{^{(}->}[d]\\
W\otimes V\otimes X\ar@{->>}[rr]&&W\otimes V\otimes A\ar@{^{(}->}[rr]&&W\otimes V\otimes Y.\\
}$$
Applying Lemma~\ref{LemSplit} to the exact sequence $?\hookrightarrow V\otimes X\tto V\otimes A$ then yields morphisms $w:\unit\hookrightarrow W$ and $ V\otimes A\to W\otimes V\otimes X$ such that the two paths from $V\otimes A$ to $W\otimes V\otimes A$ give the same morphism. By construction, every square in the diagram is commutative. We can then take $u=w\otimes v$.
\end{proof}

\begin{proof}[Proof of Theorem~\ref{Thm1}]
Consider a faithful linear monoidal functor $F:\bT_1\to \bT_2$ and a short exact sequence $\xi$ in $\bT_1$. By Lemma~\ref{LemSplit}, there exists a nonzero $\unit\to U$ in $\bT_1$ such that the induced chain map $\xi\to U\otimes \xi$ is nullhomotopic. Since $F$ is faithful, we obtain a nonzero $\unit\to F(U)$ and the induced chain map $F(\xi)\to F(U)\otimes F(\xi)$ is nullhomotopic since $F$ maps chain homotopies to chain homotopies. Again by Lemma~\ref{LemSplit}, $F(\xi)$ is exact, which concludes the proof.
\end{proof}

\begin{remark}
From the proof it is clear that Theorem~\ref{Thm1} extends from linear to additive monoidal functors.
\end{remark}

\subsection{Symmetric envelopes}

\begin{lemma}\label{LemSym}
Consider an abelian envelope $F:\bD\to\bT$ as in Definition~\ref{DefAbEnv}, such that $\bD,\bT,F$ are all symmetric. Then $F:\bD\to\bT$ is an abelian envelope in the symmetric sense, meaning 
for each symmetric tensor category $\bT_1/k$, composition with $F$ induces an equivalence 
$$\sTens(\bT,\bT_1)\;\simeq\; \sPTens^{faith}(\bD,\bT_1).$$
\end{lemma}
\begin{proof}
It follows easily from the naturality of braidings that a tensor functor $\bT\to\bT_1$ between two symmetric tensor categories $\bT,\bT_1$ is symmetric when it is symmetric restricted to a generating pseudo-tensor subcategory. By Lemma~\ref{LemTGen}, the equivalence $\Tens(\bT,\bT_1)\simeq \PTens^{faith}(\bD,\bT_1)$ therefore induces the required one.
\end{proof}

\begin{remark}\label{RemSym}
 It is clear that Lemma~\ref{LemSym} also holds within the context of braided tensor categories, or coboundary tensor categories, etc.

\end{remark}

We conclude this subsection with an observation that symmetric abelian envelopes as in Theorem~\ref{MainThmAbEnv} can be reconstructed from the tannakian principle.

\begin{notation}\label{notpi}
	Given a faithful symmetric monoidal functor $F:\bD\to \bT_1$ from a symmetric pseudo-tensor category $\bD$ to an artinian symmetric tensor category $\bT_1$, we consider the associated affine group scheme $G:=\underline{\Aut}^{\otimes}(F)$ in $\bT_1$ as in~\cite[\S 8]{Del90}. The group scheme $\pi:=\underline{\Aut}^{\otimes}(\id_{\bT_1})$ is often called the \textbf{fundamental group}, and comes with a canonical homomorphism $\epsilon:\pi\to G$. It is standard to write $\Rep_{\bT_1}(G,\epsilon)$ for the category of representations of $G$ in $\bT_1$ which restrict to the evaluation action of $\pi$ under $\epsilon$.
\end{notation}

\begin{prop} Assume that $k$ is perfect.
Consider a symmetric pseudo-tensor category $\bD$ with a faithful symmetric monoidal functor $F:\bD\to\bT_1$ to an artinian symmetric tensor category $\bT_1$. 
If $\bD$ strongly generates a tensor category $\bT$, then $\bT\simeq \Rep_{\bT_1}(G,\epsilon)$, with notation as in~\ref{notpi}.
\end{prop}

\begin{proof}
Since $\bT$ is an abelian envelope for $\bD$, there is a faithful (exact) tensor functor $H:\bT\to\bT_1$. Now Deligne~(\cite[8.17]{Del90}) tells us that $\bT$ is tensor equivalent to $\Rep_{\bT_1}(G,\epsilon)$, where $G$ is the affine group scheme associated to the functor $H$. By~\cite[Corollary~4.4.4]{PreTop}, the natural automorphisms of the functor $\bD\xrightarrow{F}\bT_1\to \mathsf{Mod}_R$ are exactly the natural automorphisms of the functor $\bT\xrightarrow{H}\bT_1\to\mathsf{Mod}_R$, for every algebra $R$ in $\Ind\bT_1$.
This shows that $G$ is also the affine group scheme associated to $F$. 
\end{proof}


\section{An intrinsic criterion}\label{SecCrit}
Let $\bD$ be a pseudo-tensor category over $k$. We give an intrinsic criterion on $\bD$ which determines whether there exists an abelian envelope in which $\bD$ is strongly generating.

In \cite{PreTop} it was observed that there is only one monoidal category which can potentially be a tensor category in which $\bD$ is strongly generating.  This is the category $\Sh(\bD,\cT_{\cU})$ below, or more precisely its subcategory of compact objects. We therefore start with a subsection which recalls the essential properties of this category.

\subsection{Sheaf categories}

 We will review $\Sh(\bD,\cT_{\cU})$ in slightly greater generality, that is for subclasses $\cV\subset\cU$. Because of the above, we are only interested in cases where $\cT_{\cV}=\cT_{\cU}$, so $\Sh(\bD,\cT_{\cU})=\Sh(\bD,\cT_{\cV})$, which is equivalent to $\cV$ being dense in the definition below. Even though they do not lead to new categories, such $\cV\subset\cU$ do show up naturally, see for instance Proposition~\ref{PropDelTP} below, or \cite[Theorem~A]{AbEnv1}. 
 

\begin{definition}\label{DefcV}
Consider a subclass $\cV\subset\cU^{ep}(\bD)$.
\begin{enumerate}
\item Denote by $\bar{\cV}\subset \cU^{ep}$ the class of the morphisms of the form
$$u_1\otimes\cdots\otimes u_n\;:\; U_1\otimes\cdots \otimes U_n\,\to\,\unit,$$
with $u_i\in\cV$. 
\item We say that $\cV$ is {\bf permutable} if for each $u:U\to\unit$ in $\cV$ and $A\in\bD$, there exists a commutative square
$$
\xymatrix{
\unit\ar[rr]^{\co_A}&&A\otimes A^\ast\\
V\ar@{-->}[u]^v\ar@{-->}[rr]&&A\otimes U\otimes A^\ast\ar[u]_{{A\otimes u \otimes A^\ast}}&&&\mbox{with $v\in\bar{\cV}$. }
}
$$

\item We say that $\cV$ is {\bf dense} if for every nonzero $u:U\to\unit$, there exists $v:V\to\unit$ in $\bar{\cV}$ and $f:V\to U$ with $v=u\circ f$.
\end{enumerate}

\end{definition}

\begin{example}\label{ExVMon}
\begin{enumerate}
\item If $\bD$ is symmetric, or more generally  the forgetful functor $\mathcal{Z}(\bD)\to\bD$ from the Drinfeld centre is essentially surjective, then every  $\cV\subset\cU^{ep}(\bD)$ is permutable.
\item If $\bT$ is a tensor category, then $\cU(\bT)$ is permutable, as the diagram in \ref{DefcV}(2) can be completed by a pullback, and we have $v\not=0$ since $A\otimes u\otimes A^\ast$ is an epimorphism.
\item For $\bU$ as in  Example~\ref{ExUni}, we find that $\cU(\bU)=\cU^{ep}(\bU)$ is {\em not} permutable. Indeed, it follows easily that for any $i\not=j\in\mZ$, the diagram
$$\unit\xrightarrow{\co_{X_j}} X_j\otimes X_{j+1}\xleftarrow{X_j\otimes \ev_{X_i}\otimes X_{j+1}} X_j\otimes X_{i+1}\otimes X_i\otimes X_{j+1}$$
can only be completed to a commutative diagram as in \ref{DefcV}(2) with $v=0$. 
\end{enumerate}
\end{example}
\subsubsection{}\label{SecShD}
Given a morphism $u:U\to\unit$ in $\bD$, recall that we write $\sigma_u$ for the sequence displayed in~\ref{ExEp}.
Following~\cite[\S 4.3]{PreTop}, we write $\Sh(\bD,\cT_{\cV})=\Sh(\bD,\cT_{\bar{\cV}})$ for the Grothendieck category of $k$-linear functors $F:\bD^{\op}\to\Vecc^{\infty}$ for which the sequence $F(\sigma_u\otimes X)$ is exact in $\Vecc^{\infty}$ for each $u\in\cV$ and $X\in\bD$. Denote by 
$$Z:\bD\to\Sh(\bD,\cT_{\cV})$$ the composition of the Yoneda embedding with the sheafification, see~\cite[\S1]{PreTop}.
 It is proved in \cite[Theorem~3.5.5 and Lemma~4.3.4]{PreTop} that, when $\cV$ is permutable, there exists an essentially unique monoidal structure on $\Sh(\bD,\cT_{\cV})$ which is biclosed and for which $Z$ admits a monoidal structure.

 \begin{lemma}\label{LemShD}
 Assume that $\cV\subset\cU$ is permutable.
 \begin{enumerate}
 \item For every $u\in \cV$, the sequence $Z(\sigma_u)$ is exact in $\Sh(\bD,\cT_{\cV})$, with $\sigma_u$ as in~\ref{ExEp}.
 \item An object in $\Sh(\bD,\cT_{\cV})$ is compact if and only if it has a presentation by objects in the image of~$Z$.
  \item 
 The monoidal functor $Z:\bD\to\Sh(\bD,\cT_{\cV})$ is fully faithful if and only if $\cV\subset\cU^{ex}$. In that case $Z$ is simply the Yoneda embedding $Y\mapsto\bD(-,Y)$.
 \item  The monoidal functor $Z:\bD\to\Sh(\bD,\cT_{\cV})$ is faithful if and only if $\cV\subset\cU^{ep}$. In the latter case, for every morphism $a: Z(X)\to Z(Y)$ there exists a morphism $q:X'\to X$ in $\bD$ such that $Z(q)$ is an epimorphism and $a\circ Z(q)$ is in the image of~$Z$.
 \item If the displayed sequence in Corollary~\ref{CorEmb}(1) is exact for each $u\in\bar{\cV}$ and $\cV\subset\cU^{ep}$, the endomorphism algebra of $\unit$ in $\Sh(\bD,\cT_{\cV})$ equals $k$.
 \end{enumerate}
 \end{lemma}
 \begin{proof}
 Part (1) is \cite[Proposition~4.3.6(ii)]{PreTop}. By \cite[Lemma~4.3.1(ii) and Definition~{3.1.1}]{PreTop} the objects in the image of $Z$ are compact. Furthermore, by \cite[Lemma~4.3.1(iv)]{PreTop}, every object in $\Sh(\bD,\cT_{\cV})$ is a quotient of a direct sum of objects in this image. This implies part (2), see for instance \cite[Lemma~1.3.1(ii)]{PreTop}.
 
 Parts (3) and (4) are contained in \cite[Lemma~4.3.1]{PreTop}.
 
Part (5) is an example of \cite[Lemma~3.4.3]{PreTop}.
 \end{proof}
 
The following proposition is a special case of \cite[Theorem~4.4.1]{PreTop}, by \cite[Lemma~4.1.3]{PreTop}.
\begin{prop}\label{propIndSh}
In the situation of Proposition~\ref{PropGF} we have a monoidal equivalence $\Ind\bT\simeq \Sh(\bD,\cT_{\cU})$ which exchanges the functors $F$ and $Z$.
\end{prop}

\begin{corollary}\label{CorCan}
 For a tensor category $\bT$, the ind-completion $\Ind\bT$ coincides with the category of $k$-linear functors $F:\bT^{\op}\to\Vecc^{\infty}$ for which the sequence $F(\sigma_u\otimes X)$,
$$0\to F(X)\to F( U\otimes X)\to F( U\otimes U\otimes X),$$
is exact, for every $X\in\bT$ and nonzero morphism $U\to\unit$.

\end{corollary}

\subsection{Main statements}\label{SecExist}
The results in this section will be proved in Section~\ref{SecPrf}.

We fix a pseudo-tensor category $\bD$ over $k$. We introduce the following convention. Consider a functor $F:\sJ\to\bD$ from a filtered category $\sJ$ and a morphism $f:X\to Y$ in $\bD$. We denote the composition $\sJ\xrightarrow{F}\bD\xrightarrow{-\otimes X}\bD$ by $F\otimes X$. By `$\colim(F)\otimes f$' we denote the morphism $\colim(F\otimes X)\to \colim(F\otimes Y)$ in $\Ind\bD$ induced from the obvious natural transformation $F\otimes X\Rightarrow F\otimes Y$.

\begin{theorem}\label{ThmThm}
The following are equivalent:
\begin{enumerate}
\item $\bD$ is strongly generating in a tensor category (which is thus its abelian envelope).
\item 
\begin{enumerate}
\item We have $\cU=\cU^{ex}$. 
\item The class $\cU$ is permutable. 
\item For every morphism $f:X\to Y$ in $\bD$ there exist nonzero $M_1,M_2\in\Ind\bD$ such that the morphisms $M_1\otimes f$ and $f\otimes M_2$ are split in $\Ind\bD$.
\end{enumerate}
\end{enumerate}
\end{theorem}

\begin{remark}\label{Remc}
Consider the following assumption 
\begin{itemize}
\item[(c')] For every $f:X\to Y$ in $\bD$ there exist nonzero $N_1\in\bD$ such that $N_1\otimes f$ is split.
\end{itemize} 
By duality, it follows that (c') implies that every morphism can also be split by tensoring on the right, in particular \ref{ThmThm}(2)(c) follows. Note that the assumption $\cU=\cU^{ep}$ (or just $\ev_X\in \cU^{ep}$ ) implies that $ X\otimes -$ is faithful for all $X\in \bD$. Using this it is possible to show that (c') together with $\cU=\cU^{ep}$ implies \ref{ThmThm}(2)(b). In conclusion, also condition \ref{ThmThm}(2)(a) and (c') imply that $\bD$ admits an abelian envelope. On the other hand (c') is not a necessary condition, see Lemma~\ref{LemGa}.
\end{remark}

If we work over a field of characteristic zero and $\bD$ is symmetric, we can omit the need to work with the ind-completion in existence criterions.

\begin{theorem}\label{ThmSym}
If $\charr(k)=0$ and $\bD$ is symmetric, the properties in Theorem~\ref{ThmThm} are equivalent with
\begin{enumerate}
\item[(3)] We have $\cU=\cU^{ex}$ and for every morphism $f:X\to Y$ in $\bD$ there exists $U\in\bD$ and morphisms $u:\unit\to U$ and $g:Y\to U\otimes X$ such that $u\otimes f=(U\otimes f)\circ g\circ f$.
\end{enumerate}
\end{theorem}

We also have an analogous result for weak abelian envelopes.

\begin{theorem}\label{TheoremV}
The following are equivalent:
\begin{enumerate}
\item $\bD$ has a faithful monoidal functor to a tensor category over $k$ satisfying (F) and (G) from \ref{PropGF}, which is therefore its weak abelian envelope.
\item 
\begin{enumerate}
\item We have $\cU=\cU^{ep}$ and the sequence in Corollary~\ref{CorEmb}(1) is exact for each $u\in\cU$. 
\item The class $\cU$ is permutable. 
\item For every morphism $f$ in $\bD$ there exist functors $\sJ_i\to \bD$ for $i=1,2$ from filtered categories $\sJ_i$ such that the colimits $Q_i$ of $\sJ_i\to \bD\to\Sh(\bD,\cT_{\cU})$ are nonzero and the morphisms $Q_1\otimes Z(f)$ and $Z(f)\otimes Q_2$ are split.
\end{enumerate}
\end{enumerate}
\end{theorem}

\begin{remark}
The tensor category in \ref{ThmThm}(1) or \ref{TheoremV}(1) is the category of compact objects in $\Sh(\bD,\cT_\cU)$, by \cite[Corollary~4.4.4]{PreTop} or the proofs below.
\end{remark}
\begin{example}
It follows that $\bU$ from Example~\ref{ExUni} does not admit an abelian envelope in which it is strongly generating (or more generally a functor as in \ref{TheoremV}(1)), since condition (2)(b) fails by~Example~\ref{ExVMon}(3).
\end{example}

\begin{remark}
It is possible, but tedious, to write versions of Theorems~\ref{ThmThm}, \ref{TheoremV} and \ref{ThmSym} for permutable $\cV\subset\cU^{ex}$ or $\cV\subset\cU^{ep}$, yielding characterisations for when $\Sh(\bD,\cT_{\cV})$ is the ind-completion of a tensor category (and then automatically equal to $\Sh(\bD,\cT_{\cU})$), without a priori knowing $\cU=\cU^{ex}$ or $\cU=\cU^{ep}$. An example of this without ind-objects (see Remark~\ref{Remc}) is given in \cite[Theorem~A]{AbEnv1}.
\end{remark}

\subsection{Some results on biclosed Grothendieck categories}

\begin{lemma}\label{LemComMon}
Consider a biclosed Grothendieck category $\bC$ such that
\begin{enumerate}
\item $\unit\in\bC$ is compact;
\item every object in $\bC$ is a quotient of a coproduct of rigid objects.
\end{enumerate}
Then the compact objects in $\bC$ are precisely the ones which have a presentation by rigid objects. In particular, the subcategory of compact objects in $\bC$ is monoidal, and every object in $\bC$ is a filtered colimit of compact objects.

\end{lemma}
\begin{proof}
If $\unit$ is compact, it follows by adjunction and cocontinuity of the tensor product that every rigid object is compact. That one can then characterise compact objects as the cokernels of morphisms between rigid objects, follows from the standard properties of compact objects, as in \cite[Lemma~1.3.1]{PreTop}. 

By assumption (and the observation that rigid objects are compact), every object is a colimit of rigid objects, so in particular a colimit of compact objects. Since cokernels of morphism between compact objects are compact, every object is a filtered colimit of compact objects.
\end{proof}

\begin{prop}\label{PropInd}
Consider a $k$-linear biclosed Grothendieck category $\bC$ such that
\begin{enumerate}
\item[(a)] $\unit\in\bC$ is compact and $k\to\bC(\unit,\unit)$ is an isomorphism;
\item[(b)] there is a faithful linear monoidal functor $F:\bD\to\bC$ from a pseudo-tensor category $\bD$ over $k$ such that:
\begin{enumerate}
\item[(G)] Every object in $\bC$ is a quotient of a coproduct of objects $F(X)$ with $X\in\bD$.
\item[(F)] For every morphism $a:F(X)\to F(Y)$ in $\bC$ there exists a morphism $q:X'\to X$ in $\bD$ such that $F(q)$ is an epimorphism and $a\circ F( q)$ is in the image of $F$.
\end{enumerate}
\end{enumerate}
The following are equivalent:
\begin{enumerate} 
\item $\bC$ is monoidally equivalent to the ind-completion of a tensor category.
\item There exists a faithful exact linear monoidal functor $\bC\to\Ind\bT$, where $\bT$ is a tensor category over some field extension $K/k$.
\item For every compact object $X\in\bC$, the endofunctors $X\otimes-$ and $-\otimes X$ are exact.
\item The full monoidal subcategory $\bC^c$ (see Lemma~\ref{LemComMon}) of compact objects in $\bC$ is a tensor category.
\item  For every morphism $f$ in $\bD$ there exist functors $F_i:\sJ_i\to \bD$ for $i=1,2$ from filtered categories $\sJ_i$ such that, with $Q_i:=\colim (F\circ F_i)$, the endofunctors $Q_1\otimes-$ and $-\otimes Q_2$ of $\bC$ are faithful on objects and the morphisms $Q_1\otimes F(f)$ and $F(f)\otimes Q_2$ are split.

\end{enumerate}
If the conditions are satisfied, then $\Ind(\bC^c)\simeq\bC$.
\end{prop}

\begin{proof}
Clearly, (1) implies (2) via the identity functor. That (2) implies (3) is clear, see Lemma~\ref{LemIndFaith}(2).

We show that (3) implies (4). It suffices to show that every kernel and cokernel of a morphism between rigid objects is again rigid. Indeed, by Lemma~\ref{LemComMon} the statement concerning cokernels shows that compact and rigid objects coincide. The statement about kernels further shows that the latter category is abelian. As we will show that the left and right dual of the cokernel of $f:X\to Y$ for rigid $X,Y\in\bC$ are in fact the kernel of $Y^\ast\to X^\ast$ and ${}^\ast Y\to{}^\ast X$, the statement about kernels will follow from our proof about cokernels. Since the proof of the left and right dual are analogous, we only look at the left dual.

By Lemmata~\ref{LemComMon} and~\ref{LemNN}, the tensor product on $\bC$ is exact. We set $Z:=\coker(f)$ and $D:=\ker(f^\ast)\simeq [Z,\unit]_r$ and take arbitrary $M,N\in\bC$. Starting from the exact sequence $X\to Y\to Z\to 0$ we find an exact sequence 
$$0\to \bC(M,[Z,N]_r)\to\bC(M,N\otimes Y^\ast )\xrightarrow{( N\otimes f^\ast)\circ-} \bC(M, N\otimes X^\ast).$$
Hence 
$$0\to [Z,N]_r\to N\otimes Y^\ast \xrightarrow{N\otimes f^\ast} N\otimes X^\ast $$
is exact.
Since we already know that $N\otimes -$ is (left) exact, it follows that the canonical morphism
$$N\otimes D\;\to\;[Z,N]_r$$
is an isomorphism, for each $N\in \bC$. It follows from \cite[Proposition~2.3]{Del90} that $D\simeq Z^\ast$.

That (4) implies (1) is a direct consequence of the standard fact that any locally finitely presentable Grothendieck category for which the subcategory of compact objects is abelian is the ind-completion of the latter subcategory, see~\cite[Proposition 2]{Ro} for instance.

It now remains to prove that (5) is equivalent to the other properties. 

We prove that (1) implies (5). Assume that $\bC\simeq\Ind\bT$ for some tensor category $\bT$ as in (1). Since all rigid objects in $\Ind\bT$ are compact (\cite[Lemma~1.3.7]{AbEnv1}), it follows that $F:\bD\to\Ind\bT$ takes values in $\bT$.
 It is therefore proved in \cite[Lemma~4.1.3 and Lemma~4.1.6(ii)]{PreTop} that every injective object $I\in\Ind\bT$ is a filtered colimit as in (5). That $I\otimes-$ is faithful follows from Lemma~\ref{LemIndFaith}(2). We prove the splitting property in (5). Since $\bT$ is an abelian subcategory, the kernel and image of $F(f)$ are in $\bT$ and hence rigid. By adjunction $I\otimes V$ is injective for every rigid object $V$. It then follows easily that $I\otimes F(f)$ is split. The argument for $-\otimes I$ is identical.

Finally, we prove that (5) implies (3).
Take $Q_1, Q_2$ as in (v). By Lemma~\ref{LemNN}, $Q_1\otimes-$ and $-\otimes Q_2$ are exact endofunctors of $\bC$. By assumption that they be faithful on objects, they are thus faithful. 
Now take an arbitrary compact object $X$ in $\bC$. By combination of assumptions (F) and (G), it is of the form $\coker F(f)$ for some morphism $f$ in $\bD$. For $Q_1\in\bC$ as in (v), $Q_1\otimes X $ is therefore a direct summand of an object $Q_1\otimes X'$ with $X'$ rigid. As $Q_1\otimes-$ is exact, it follows that $Q_1\otimes X\otimes-$  is exact. Since $Q_1\otimes-$ is also faithful, $X\otimes-$ is exact. The same reasoning leads to the corresponding conclusion on $-\otimes X$.
\end{proof}

\begin{lemma}\label{fMN}
For rigid $M,N$ in a biclosed Grothendieck category $\bC$ and $f\in\bC(M, N)$, the following are equivalent:
\begin{enumerate}
\item $f:M\to N$ is an epimorphism;
\item $(  Z\otimes f^\ast):\,    Z\otimes N^\ast\to    Z\otimes M^\ast$ is a monomorphism for every $Z\in\bC$.
\end{enumerate}
\end{lemma}
\begin{proof}
Property (2) means that composition $(  Z\otimes f^\ast)\circ-$ yields a monomorphism
$$\bC(W, Z\otimes N^\ast)\;\hookrightarrow\; \bC(W,Z\otimes M^\ast),\quad\mbox{for every $W,Z\in\bC$.}$$
We can use adjunction to rewrite this as a monomorphism
$$\bC(N\otimes W, Z)\;\hookrightarrow\; \bC( M\otimes W, Z),\quad\mbox{for every $W,Z\in\bC$,}$$
which states that $ f\otimes W$ is an epimorphism for every $W\in\bC$. Since the tensor product is right exact, the latter is equivalent to (1).
\end{proof}

\subsection{Proofs}\label{SecPrf}


\begin{lemma}\label{Lemmono}
For a permutable $\cV\subset\cU^{ep}$,
consider $v: V\to\unit$ in $\bar{\cV}$ and $N\in \Sh(\bD,\cT_\cV)$. Then $ N\otimes Z(v^\ast):N\to N\otimes Z^\ast$ is a monomorphism in $\Sh(\bD,\cT_\cV)$.
\end{lemma}
\begin{proof}
By Lemma~\ref{fMN}, it is sufficient to show that $Z$ maps $v$ to an epimorphism in $\Sh(\bD,\cT_{\cV})$. Since $Z(u)$ with $u\in\cV$ is an epimorphism, see Lemma~\ref{LemShD}, the same is true for $v\in\bar{\cV}$.
\end{proof}

\begin{lemma}\label{Lemmono2}
Assume $\cU=\cU^{ep}$ and $\cU$ is permutable. Consider a nonzero morphism $a:\unit\to M$ in $\Sh(\bD,\cT_{\cU})$ with $M$ a filtered colimit of some $\sJ\to \bD\to\Sh(\bD,\cT_{\cU})$.
For every $N\in\Sh(\bD,\cT_{\cU})$, we have that $N\otimes a$ is a monomorphism.
\end{lemma}
\begin{proof}
Assume first that $M=Z(X)$ for $X\in\bD$. By Lemma~\ref{LemShD}(4) there exists a morphism $b:X\to Y$ in $\bD$ such that $Z(b)$ is a monomorphism and $Z(b)\circ a=Z(c)$ for some (by construction non-zero) $c:\unit\to Y$. Now $N\otimes Z(c)$ is a monomorphism by Lemma~\ref{Lemmono}, so also $N\otimes a$ is a monomorphism.
Since filtered colimits of monomorphisms are monomorphisms, we can derive the general claim from the above special case.
\end{proof}

\begin{proof}[Proof of Theorem~\ref{TheoremV}]
To prove that (1) implies (2), consider a faithful monoidal functor $u:\bD\to\bT$ to a tensor category $\bT$ over $k$, satisfying (F) and (G). Condition (a) follows from Corollary~\ref{CorEmb}(1). Condition (b) follows easily from Example~\ref{ExVMon}(2). Condition (c) follows from Proposition~\ref{PropInd} applied to the functor $\bD\to \Ind\bT\simeq \Sh(\bD,\cT_{\cU})$ (see~Proposition~\ref{propIndSh}).

Now we prove that (2) implies (1). By Lemma~\ref{LemShD}, it suffices to prove that $\Sh(\bD,\cT_{\cU})$ is the ind-completion of a tensor category over $k$. 
Applying Proposition~\ref{PropInd} to $Z:\bD\to \Sh(\bD,\cT_{\cU})$, it suffices to show we can take the $Q_i$ in (2)(c) to be such that $Q_1\otimes-$ and $-\otimes Q_2$ are faithful on objects. By replacing $Q_1$ by some $Z({}^\ast X)\otimes M_1$ for an appropriate $X\in\bD$ we can assume, without loss of generality, that there is a nonzero morphism $\unit\to Q_1$. That $Q_1\otimes-$ is faithful on objects now follows from Lemma~\ref{Lemmono2}, and we can deal with $-\otimes Q_2$ similarly.
\end{proof}

\begin{proof}[Proof of Theorem~\ref{ThmThm}]
That (1) implies (2)(a) follows from Corollary~\ref{CorEmb}(2). That (1) implies (2)(b) is a special case of the result in~Theorem~\ref{TheoremV}.
To prove condition (2)(c), consider a fully faithful embedding $u:\bD\to\bT$ such that $\bD$ is strongly generating in the tensor category~$\bT$. We need to show that for every morphism $f:X\to Y$ in $\bD$, there exist nonzero $M_1,M_2\in\Ind\bD$ such that the morphisms $M_1\otimes f$ and $f\otimes M_2$ are split in~$\Ind\bD$.
Considering the functor $\bD\to\Ind\bT$, we find 
functors $F_i:\sJ_i\to\bD$ and $Q_i$ as in~Proposition~\ref{PropInd}(5). Set $M_i:=\colim F_i\in\Ind\bD$, so that $(\Ind\,u)(M_i)=Q_i$. Since $\Ind\,u$ is fully faithful, the morphism $Q_1\otimes u(Y)\to Q_1\otimes u(X)$ which splits $Q_1\otimes u(f)$ in $\Ind\bT$ yields a morphism $M_1\otimes Y\to M_1\otimes X$ which splits $M_1\otimes f$ in $\Ind\bD$. The same reasoning works for $f\otimes M_2$.


By Theorem~\ref{TheoremV} and~Lemma~\ref{LemShD}(3), to show (2) implies (1) we only need to argue that the condition in (2)(c) implies the corresponding condition in \ref{TheoremV}(2)(c). This follows from faithfulness of $Z$ and the fact that $Z$ sends compact objects to compact objects, from which we can derive that the extension of $Z$ to $\Ind\bD$ does not send non-zero objects to zero.
\end{proof}

\begin{proof}[Proof of Theorem~\ref{ThmSym}]
To show \ref{ThmThm}(1) implies (3), assume $\bD$ is strongly generating in a  tensor category $\bT$. For a given $f$, by Corollary \ref{Corfu}, there exists a morphism $u':\unit\to U'$ in $\bT$ and a morphism $g':Y\to U'\otimes X$ in $\bT$ such that $u'\otimes f=(U'\otimes f)\circ g'\circ f$. Let $U$ be an object in $\bD$ that has $U'$ as a subobject. Then (3) holds for the morphisms $u:\unit\xrightarrow{u'} U'\subset U$ and $Y\xrightarrow{g'} U'\otimes X\subset U\otimes X$. 

Now we prove that (3) implies \ref{ThmThm}(2).
Fix $f,\,u$ and $g$ as in (3). Firstly, we observe that, since $\cU=\cU^{ex}$ (so in particular $\cU=\cU^{ep}$), the morphisms $u^{\otimes i}:\unit\to U^{\otimes i}$ are not zero (they are monomorphisms). Since $\mathrm{char}(k)=0$, these morphisms actually factor through $\unit\to \Sym^i U$, so the latter are non-zero too. Next consider the directed system 
$$\unit\xrightarrow{u} U\to \cdots\to\Sym^{i-1}U\to \Sym^iU\to \cdots,$$
where each arrow is a composition of $u\otimes\Sym^{i-1}U $ with projection onto the direct summand $\Sym^iU$. We can consider the filtered colimit $S_u:=\varinjlim_{i\in\mN} \Sym^i U$, which is non-zero since the canonical $\unit\to S_u$ is a monomorphism by the observations from the start of this paragraph.

It suffices to show that $S_u\otimes f$ is split in $\Ind\bD$, or equivalently in $\Sh(\bD,\cT_{\cU})$. We will prove a slightly stronger statement. By \cite[Example~7.12]{Del90}, $S_u$ has the structure of a commutative algebra in $\Sh(\bD,\cT_{\cU})$, with unit $\unit\to S_u$ induced from the above directed system. Let $h:S_u\otimes Y\to S_u\otimes X$ be the morphism of $S_u$-modules in $\Sh(\bD,\cT_{\cU})$ corresponding to the image of $g$ under
$$\bD(Y,U\otimes X)\;\to\;\varinjlim_{i\in\mN}\bD(Y, \Sym^iU\otimes X)\;\simeq\;\Hom_{S_u}(S_u\otimes Y,S_u\otimes X).$$
Composing the assumed identity between $f$ and $g$ with $(U\to S_u)\otimes Y$ shows that $(S_u\otimes f)=(S_u\otimes f)\circ h\circ (S_u\otimes f)$. In particular, $S_u\otimes f$ is split, even as an $S_u$-module morphism.  
\end{proof}



\section{Tannakian envelopes}\label{SecTann}

We will assume $k$ is algebraically closed throughout this section. 
Every representation of an affine group scheme is assumed to be finite dimensional and algebraic.

\subsection{Observable and epimorphic subgroups}

In this subsection, we do not claim originality, but derive some consequences of the results in~\cite{BBHM} in the form needed for the following subsection. Most of the statements can also be extracted from~\cite{Gr} and we use the same terminology as {\it loc. cit.}

\subsubsection{}\label{DefQuo}
By `algebraic group' we mean an affine group scheme of finite type (with finitely generated coordinate algebra).
Recall that for a closed subgroup $G_2$ of an algebraic group $G_1$, the quotient scheme $G_1/G_2$ exists and is of finite type, see \cite{DG}. From the quotient property we have the isomorphism
$$k[G_1]^{G_2}\simeq \Hom(G_1,\mA^1)^{G_2}\simeq \Gamma(G_1/G_2,\cO_{G_1/G_2}).$$

\begin{theorem}[Byalinicki-Birula, Hochschild and Mostow]\label{BBHM}
	The following are equivalent conditions on a closed subgroup $G_2$ of an algebraic group $G_1$.
	\begin{enumerate}
		\item The scheme $G_1/G_2$ is quasi-affine.
		\item Any finite dimensional representation of $G_2$ embeds $G_2$-equivariantly into a $G_1$-representation.
		\item The group $G_2$ is the stabiliser of a vector in a $G_1$-representation.
	\end{enumerate}
	Such subgroups $G_2<G_1$ are called {\bf observable}.
\end{theorem}
\begin{proof}
	The equivalence of (1) and (2) is in \cite[Theorem~4]{BBHM}. The equivalence of (1) and (3) is \cite[Theorem~8]{BBHM}.
	
	Note that \cite{BBHM} only deals with smooth groups (varieties), but the same results hold for general algebraic groups. This result and extension to supergroups will appear in future work of the first author.
\end{proof}

\begin{remark}\label{RemObsSG}
	Condition~\ref{BBHM}(2) is equivalent to the condition that (the closure under taking direct summands of) the essential image of $\Rep G_1\to \Rep G_2$ is strongly generating in $\Rep G_2$.
\end{remark}

Recall that a scheme $X/k$ is anti-affine if $\Gamma(X,\cO_X)=k$. The following classical observation is a clear counterpart to Theorem~\ref{BBHM}.

\begin{theorem}\label{ThmEpi}
	The following are equivalent conditions on a closed subgroup $G_2$ of an algebraic group $G_1$.
	\begin{enumerate}
		\item The scheme $G_1/G_2$ is anti-affine.
		\item The restriction tensor functor $\Rep G_1\to \Rep G_2$ is fully faithful.
		\item In any $G_1$-representation, the $G_1$-invariants and $G_2$-invariants coincide.
	\end{enumerate}
	Such subgroups $G_2<G_1$ are called {\bf epimorphic}.
\end{theorem}
\begin{proof}
	We freely use that fully faithfulness of tensor functors between tensor categories need only be verified on morphism spaces $\Hom(X,\unit)$ or $\Hom(\unit,X)$.
	By \cite[Chapter 3]{Jantzen} we have
	$$\Hom_{G_2}(\res^{G_1}_{G_2}V,k)\simeq\Hom_{G_1}(V,k[G_1]^{G_2}),$$
	for arbitrary $V\in \Rep G_1$. By \ref{DefQuo}, we find that (1) and (2) are equivalent.
	For any $V\in \Rep G$ we have $V^G\simeq \Hom_G(k,V)$, so (2) and (3) are equivalent.
\end{proof}

\begin{theorem}\label{ThmHull}
	For a closed subgroup $H$ of an algebraic group $K$, consider the following properties on closed subgroups $H'<K$ containing $H$.
	\begin{enumerate}
		\item $H'$ is minimal under the condition that $H'<K$ is observable (minimal such that $K/H'$ is quasi-affine);
		\item $H'$ is maximal under the condition that $H<H'$ is epimorphic (maximal such that $H'/H$ is anti-affine);
		\item $H<H'$ is epimorphic and $H'<K$ is observable.
	\end{enumerate}
	All three conditions are equivalent and there exists a unique $H<H'<K$ which satisfies them. That group $H'$ is the {\bf observable hull} of $H$ in $K$.
\end{theorem} 
\begin{proof}
	It follows easily from Theorem~\ref{BBHM}(3) that the intersection of two observable subgroups is again observable. Clearly the collection of $H<L<K$ such that $K/L$ is quasi-affine is not empty, since we can take $K=L$. Since $K$ is noetherian, the intersection of all such $L$ is a closed subgroup $H'<K$ satisfying condition (1). In particular we found that there exists a unique $H'$ which satisfies (1).
	
	Let $H'$ be the group constructed in the above paragraph, or equivalently the unique one satisfying condition (1). We will show that $H'/H$ is anti-affine, in other words that $H'$ satisfies (3).
	By Theorem~\ref{ThmEpi}, we need to show that for every finite dimensional $V\in \Rep H'$, the inclusion $V^{H'}\subset V^H$ is an equality. By Theorem~\ref{BBHM}, we know that each $V$ is an $H'$-submodule of a $K$-module $W$. Hence it suffices to show that for every finite dimensional $W\in \Rep K$, the inclusion $W^{H'}\subset W^H$ is an equality. Consider $w\in W^H$ and denote its stabiliser in $K$ by $L>H$. Then $L$ is an observable subgroup of $K$ by Theorem~\ref{BBHM} and hence $H'<L$. In particular, $w\in W^{H'}$ as desired.


	Now we prove that a group which satisfies (3) also satisfies (2). For $H'$ as in (3) assume there exists $H'< H''<K$ such that $H''/H$ is anti-affine. Consider the morphism
	$$H''/H\tto H''/H'\hookrightarrow K/H'.$$
	The scheme $K/H'$ is quasi-affine and hence embeds into an affine scheme. The above morphism then factors via $\Spec(\Gamma(H''/H,\cO_{H''/H})=\Spec(k)$, using that $H''/H$ anti-affine. This shows that $H'=H''$.
	
	Theorem~\ref{ThmEpi}(3) shows that the product of two epimorphic subgroups is again epimorphic. This implies that there can be at most one group $H'$ which satisfies (2). This uniqueness now allows us to conclude by the above that any group which satisfies (2) also satisfies (1). This completes the proof.
\end{proof}

\subsection{Characterisation of tannakian envelopes}

\begin{theorem}\label{ConjTann}
Consider an affine group scheme $G$ over $k$ and a pseudo-tensor subcategory $\bD\subset\Rep G$.
\begin{enumerate}
\item The pseudo-tensor category $\bD$ admits an abelian envelope which is of the form $\Rep H$ for an affine group scheme $H/k$.
\item The inclusion $\bD\subset\Rep G$ is an abelian envelope if and only if $\bD$ is strongly generating.
\item The full  subcategory $\Coker \bD$ of $\Rep G$ comprising all objects with a presentation by objects in $\bD$ is an abelian subcategory as well as a pseudo-tensor subcategory of $\Rep G$, so in particular a tensor category.
\item Consider a faithful representation $V$ of $G$ (so $G$ is in particular an algebraic group). Then the abelian envelope of the pseudo-tensor subcategory of $\Rep G$ generated by $V$ is $\Rep G'$, with $G'<GL(V)$ the observable hull of $G$ in $GL(V)$.\end{enumerate}
\end{theorem}


We start the proof by investigating a specific case. 

\begin{prop}\label{PropExGLV}
Let $G$ be an algebraic group.
\begin{enumerate}
\item A representation $V\in\Rep G $ is generating if and only if it is faithful.
\item The following properties are equivalent for a faithful representation $V\in\Rep G $:
\begin{enumerate} 
\item $V$ is strongly generating;
\item the representation $G<GL(V)$ makes $G$ an observable subgroup in $GL(V)$;
\item there exists no closed subgroup $G_1<GL(V)$ strictly containing $G$ such that $G<G_1$ is epimorphic.
\end{enumerate}
\end{enumerate} 
\end{prop}
\begin{proof}
Part (1) is classical, see~\cite[\S 2]{DM}.
Now let $V$ be faithful. By Theorem~\ref{BBHM}, $G<GL(V)$ is observable if and only if every $G$-representation embeds into a $GL(V)$-representation. By Example~\ref{ExVsg}, every $GL(V)$-representation embeds into a polynomial in $V,V^\ast$. Hence, $G<GL(V)$ is observable if and only if every $G$-representation is a submodule of a polynomial in the $G$-representations $V,V^\ast$, which shows equivalence between (2)(a) and (2)(b). Equivalence between (2)(b) and (2)(c) follows from Theorem~\ref{ThmHull}.
\end{proof}

\begin{example}\label{ExGP}
Consider a non-trivial epimorphic subgroup $P<G$ (for instance $G$ is reductive and $P$ is a parabolic subgroup) and let $\bD\subset\Rep P$ be $\Rep G$ considered as a pseudo-tensor subcategory via the fully faithful tensor functor $\Rep G\to\Rep P$, see Theorem~\ref{ThmEpi}. By Proposition~\ref{PropExGLV}(1), $\bD$ is a generating pseudo-tensor category. However, clearly $\Rep P$ is not the abelian envelope of $\Rep G$.
\end{example}

\begin{proof}[Proof of Theorem~\ref{ConjTann}]
We start by proving part (4). Let $\bD$ be the pseudo-tensor subcategory generated by $V$ (comprising all direct summands of polynomials in $V$ and $V^\ast$). By Theorems~\ref{ThmHull} and~\ref{ThmEpi}, we can interpret $\Rep G'$ as a pseudo-tensor subcategory of $\Rep G$ and by construction, we have $\bD\subset\Rep G'$. Since $G'$ is observable in $GL(V)$, it follows from Proposition~\ref{PropExGLV}(2) that $\bD$ is strongly generating in $\Rep G'$. Part (4) thus follows from Theorem~\ref{MainThmAbEnv}.

We stay with the above setting. The exact restriction functor $\Rep G'\to \Rep G$ along $G<G'$ is fully faithful. Hence $\Rep G'$ is an abelian subcategory of $\Rep G$. Since $\bD$ is strongly generating in $\Rep G'$ it follows that $\Rep G'$ is precisely $\Coker\bD$ (where there is no ambiguity in which category we take cokernels).

Now let $G$ be an arbitrary affine group scheme. For any representation $V$, consider its kernel $N\lhd G$. The representation category $\Rep G/N$ of the algebraic group $G/N$ is equivalent with the topologising (so in particular abelian) subcategory of representations in $\Rep G$ on which $N$ acts trivially. Let $\bD$ be the pseudo-tensor category generated by $V$. By the above paragraph $\Coker \bD$ is an abelian subcategory of $\Rep G/N$, so also an abelian subcategory of $\Rep G$.

Now let $\bD$ be an arbitrary pseudo-tensor subcategory of $\Rep G$ and consider $\Coker\bD$. We claim that $\Coker \bD$ is closed in $\Rep G$ under taking kernels and cokernels. Since any morphism between objects in $\Coker\bD$ is a morphism between objects in $\Coker\bD'$ for $\bD'\subset\bD$ a finitely generated pseudo-tensor subcategory, the claim follows from the above paragraph. Hence $\Coker\bD$ is an abelian subcategory of $\Rep G$. To show that it is a pseudo-tensor subcategory, we only need to argue that it is closed under taking duals. However, the dual of a cokernel of a morphism in $\bD$ is the kernel of a morphism in $\bD$. Since $\Coker\bD$ is an abelian subcategory, the latter kernel is indeed contained in $\Coker\bD$. This concludes the proof of (3).

By Theorem~\ref{MainThmAbEnv}, $\Coker\bD$ is the abelian envelope of $\bD$. Clearly, we have $\Coker\bD=\Rep G$ if and only if $\bD$ is strongly generating. This proves part (2). That $\Coker\bD$ is of the form $\Rep H$ follows from the fact that we have a tensor functor $\Coker\bD\hookrightarrow \Rep G\to\Vecc$ and the classical theory of tannakian reconstruction, see \cite{DM, EGNO}.
\end{proof}


\section{Extension of scalars of tensor categories}\label{SecExt}

\subsection{Definition and connection with abelian envelopes}

Consider a field extension $K/k$ and a pseudo-tensor category $\bD/k$.

\subsubsection{}\label{DefKk}The naive extension of scalars $K\otimes\bD$ is a category with the same objects as $\bD$ and morphism $K$-spaces given by $K\otimes_k\bD(-,-)$. Its pseudo-abelian envelope $K\stackrel{.}{\otimes}\bD$ is canonically a pseudo-tensor category over $K$.

For a tensor category $\bT$ over $k$, we consider the biclosed Grothendieck category $(\Ind\bT)_K$ of $K$-modules in $\Ind\bT$. We denote by $\bT_K$ the (monoidal by Lemma~\ref{LemComMon}) full subcategory of compact objects in $(\Ind\bT)_K$. We have an obvious fully faithful monoidal functor
$$K\stackrel{.}{\otimes} \bT\;\to\;\bT_K,\quad X\mapsto K\otimes X,$$
which is strongly generating.

We write $\cV\subset\cU^{ex}(K\stackrel{.}{\otimes}\bT)$ for the class of morphisms $1\otimes u$ in $K\stackrel{.}{\otimes}\bT$ with $u\in\cU(\bT)$. Since $\cU(\bT)$ is permutable, see Example~\ref{ExVMon}(2), we know $\cV$ is permutable.
\begin{prop}\label{PropExtSca}
The following conditions are equivalent:
\begin{enumerate}
\item The monoidal category $\bT_K$ is a tensor category over $K$.
\item For every object $X$ in $\bT_K$, the endofunctors $X\otimes-$ and $-\otimes X$ are exact as endofunctors of $(\Ind\bT)_K$.
\item The pseudo-tensor category $K\stackrel{.}{\otimes}\bT$ admits an abelian envelope in which it is strongly generating, and $\cV$ is dense.
\item The category $\Sh(K\stackrel{.}{\otimes}\bT,\cT_\cV)$ is the ind-completion of a tensor category.
\end{enumerate}
It the conditions are satisfied, the abelian envelope from (3) is $\bT_K$ and $\Ind(\bT_K)\simeq (\Ind\bT)_K$.\end{prop}
\begin{proof}
That statements (3) and (4) are equivalent is a special case of~\cite[Corollary~4.4.2]{PreTop}. Furthermore, it follows from \cite[Theorem~4.4.1]{PreTop} that the abelian envelope in (3) then has ind-completion $\Sh(K\stackrel{.}{\otimes}\bT,\cT_\cV)$. Proposition~\ref{PropInd} applied to $\bC=(\Ind\bT)_K$ shows that (1) and (2) are equivalent and they are equivalent with $(\Ind\bT)_K$ being the ind-completion of a tensor category, where the latter tensor category then must be $\bT_K$.

In conclusion, all statements in the proposition follow if we prove that $(\Ind\bT)_K$ and  $\Sh(K\stackrel{.}{\otimes}\bT,\cT_\cV)$ are monoidally equivalent. By uniqueness of the monoidal structure, see~\ref{SecShD}, it is sufficient to show there exists an equivalence between $(\Ind\bT)_K$ and  $\Sh(K\stackrel{.}{\otimes}\bT,\cT_\cV)$ which yields a commutative diagram with the canonical functors out of $K\stackrel{.}{\otimes}\bT$. By \cite[Example~3.6.4 and Theorem~3.6.2]{PreTop} and Corollary~\ref{CorCan}, we find that $ \Sh(K\stackrel{.}{\otimes}\bT,\cT_\cV)$ is equivalent to the category of $k$-linear functors from $K$, as a 1-object category, to $\Ind\bT$. Hence  $(\Ind\bT)_K\simeq \Sh(K\stackrel{.}{\otimes}\bT,\cT_\cV)$. That the equivalence from \cite[Example~3.6.4 and Theorem~3.6.2]{PreTop} respect the functors from $K\stackrel{.}{\otimes}\bT$ follows easily.
\end{proof}



\begin{example}\label{ex:ext}
	In the following cases, the properties in Proposition~\ref{PropExtSca} are satisfied.
	\begin{enumerate}
		\item If $K$ is a finite separable extension of $k$. Indeed, in that case every compact object $X$ in $(\Ind\bT)_K$ is a direct summand of the rigid object $K\otimes X_0$ (see for instance p.~603 in \cite{Del14}), where $X_0\in\bT$ is the underlying object of the $K$-module $X$.
		\item If $k$ is algebraically closed and $\bT$ is artinian, see \cite[Lemma~2.2]{EO}. We will generalise this in Corollary~\ref{cor:perf} below.
		\item If $\bT$ is tannakian, in the sense of \cite{Del90} that it fibres over a $k$-scheme. Indeed, tannakian categories are artinian, so by Proposition~\ref{prop:abs} below we only need to consider the case where $K/k$ is finite, which is verified in~\cite[5.10]{Del14}. 
		\item If $\bT$ is the (non-artinian) tensor category in \cite[2.19]{Del90}, for arbitrary $k$ of characteristic~$0$. In this case we have in fact $K\stackrel{.}{\otimes}\bT\simeq \bT_K$.
	\end{enumerate}
\end{example}

\begin{remark}
Assume that $K/k$ is a finite extension. In this case $\bT_K$ is easily seen to be equivalent to the abelian category of $K$-modules in $\bT$. We will use this observation freely.
\end{remark}
\subsection{Absolute tensor categories}

\begin{definition}
A tensor category $\bT$ over $k$ is {\bf absolute} if $\bT_K$ is a tensor category for every field extension $K/k$.
\end{definition}

\begin{remark}
\begin{enumerate}
\item We don't know any examples of tensor categories that are not absolute. All tensor categories are absolute if and only if extensions of scalars along $k\to k[t]/(t^p-\lambda)$ (for $\mathrm{char}(k)=p>0$ and $\lambda\in k\backslash k^p$) of tensor categories always yield tensor categories.
\item Assume that $\bT$ is artinian, then $\bT_K$ is automatically artinian over $K$, as follows for instance from~\cite[Theorem 1.9.15]{EGNO}.
\end{enumerate}

\end{remark}

\begin{prop}\label{prop:abs}
If $\bT$ is artinian, then the following are equivalent.
\begin{enumerate}
\item $\bT$ is absolute.
\item $\bT_{K}$ is a tensor category for each finite extension $K/k$.
\item $\bT_{\overline{k}}$ is a tensor category, where $\overline{k}$ is an algebraic closure of $k$.
\end{enumerate}
\end{prop}

\begin{proof}	
It is clear that (1) implies (2). To show that (2) implies (3), by~Proposition~\ref{PropExtSca} it suffices to show that every object in $\bT_{\overline{k}}$ is rigid.
Every object $Z$ in $\bT_{\overline{k}}$ is the cokernel of a morphism $\overline k\otimes X\xrightarrow{f} \overline k\otimes Y$ in ${\overline{k}}\otimes\bT$. Write $f$ as a finite sum $\Sigma_i g_i\otimes h_i$ with $g_i\in \overline{k}$ and $h_i \in {\bT}(X,Y)$ and let $K$ be some finite extension of $k$ containing $g_i$ for all $i$. Then $f=\overline{k}\otimes_K f'$ for some morphism $f':K\otimes X\to K\otimes Y$ in~$\bT_{K}$. 
Since $\bT_{K}$ is a tensor category, the cokernel of $f'$ in $\bT_{K}$ is rigid. Hence $Z=\overline k\otimes_K \coker(f')$ is rigid in~$\bT_{\overline{k}}$.

Finally, we prove that (3) implies (1). Suppose $\bT_{\overline{k}}$ is a tensor category and let $K$ be a field extension of $k$ with algebraic closure $\overline K$. 
Since $\bT_{\overline{k}}$ must be absolute by~\cite[Lemma~2.2]{EO}, certainly $(\Ind\bT)_{\overline K}\simeq ((\Ind\bT)_{\overline{k}})_{\overline K}$ is the ind-completion of a tensor category. 
There is a faithful exact linear monoidal functor from $(\Ind\bT)_{K}$ to $(\Ind\bT)_{\overline K}$, so $\bT_K$ is a tensor category by Proposition~\ref{PropInd}.
\end{proof}

In light of~Example~\ref{ex:ext}(1), we get the following corollary:

\begin{corollary}\label{cor:perf}
	Every artinian tensor category over a perfect field is absolute.
\end{corollary}


\begin{lemma}
If a tensor category $\bT$ over $k$ admits a tensor functor to an absolute tensor category $\bT'$ over $k$, then $\bT$ is also absolute.
\end{lemma}
\begin{proof}
It is well-known that $\Ind\bT\to\Ind\bT'$ remains exact and faithful. Fix a field extension $K/k$. The canonical functor $(\Ind\bT)_K\to(\Ind\bT')_K$ is exact and faithful, since the functors forgetting the $K$-action are exact and faithful. We can therefore apply Proposition~\ref{PropInd}.
\end{proof}



\subsection{Super-tannakian categories}

Recall that a super-tannakian category is a symmetric tensor category, over a field $k$ of characteristic different from 2, which admits a tensor functor to the biclosed Grothendieck category of super $R$-modules for a super-commutative $k$-algebra~$R$. By taking a maximal ideal in $R$, we can assume that $R$ is in fact a field extension of~$k$.

The following theorem extends \cite[Theorem~5.4]{Del14} from tannakian to super-tannakian categories. Note that \cite[Theorem~5.4]{Del14} only deals with finite field extensions, which is sufficient by Proposition~\ref{prop:abs}.
\begin{theorem}\label{th:superabs}
Every super-tannakian category is absolute. Moreover, $\bT_K$ remains super-tannakian for every field extension $K/k$ and super-tannakian category $\bT$ over $k$.
\end{theorem}

If $\bT$ is a super-tannakian category over $k$ and $\bT_K$ is a tensor category for some field extension $K/k$, then $\bT_K$ is automatically super-tannakian. Indeed, if $\bT$ fibres over $R$, then $\bT_K$ fibres over $K\otimes_kR$.
It thus suffices to show that every super-tannakian category is absolute. By Corollary~\ref{cor:perf} we already know that tensor categories over fields of characteristic zero are absolute, so we can assume that $k$ is a field of characteristic $p>2$ in the proof of~Theorem~\ref{th:superabs}.

\begin{lemma}\label{lem:super}
Let $K$ be a field and set $R:=K[\xi]/(\xi^p)$. 
Consider the monoidal category $\bC$ of finitely generated $\mZ/2$-graded $R$-modules, with symmetric structure given by the Koszul sign rule. If for $M\in\bC$ and every $n\in\mN$, the image of every element of $\mZ S_n$ on $\otimes^n_RM$ is either zero or a faithful module, then $M$ is a free $R$-module.
\end{lemma}

\begin{proof}
We can write $M$ as 
$$M_1\oplus M_2\oplus \cdots \oplus M_m\oplus \Pi M_{m+1}\oplus \cdots\oplus \Pi M_{m+n}$$
where each $M_i$ is a cyclic $R$-module (in even degree) and $\Pi M_j$ denotes it is contained in odd degree. We have the induced decomposition for each tensor power $M^{\otimes_R^r}$, and consider $r=m+(p-1)n$.

We let $a=\sum(-1)^{\mathrm{sign}(\sigma)}\sigma$ be the skew symmetriser in $\mZ S_{m+(p-1)n}$. It follows easily that $a$ acts non-trivially on a given term in the canonical decomposition of $M^{\otimes (m+(p-1)n)}$ if and only if the term is a permutation of
$$M_1\otimes_R M_2\otimes_R \cdots \otimes_R M_m\otimes_R (\Pi M_{m+1})^{\otimes_R p-1}\otimes_R\cdots\otimes_R (\Pi M_{m+n})^{\otimes_R p-1}. $$
Hence the image of the action of $a$ is a nonzero module which can only be faithful if each $M_i$ is so. This shows that $M$ is free.
\end{proof}

\begin{proof}[Proof of Theorem \ref{th:superabs}]
Let $\bT$ be a super-tannakian category over $k$, where $k$ has characteristic $p>2$. 
By Proposition~\ref{prop:abs}, iteration and Example~\ref{ex:ext}(1), it suffices to show that $\bT_l$ is a tensor category whenever $l$ is purely inseparable over $k$ of degree $p$.
Let $K/k$ be a field extension such that there exists a faithful exact monoidal $k$-linear functor $\omega:\bT\to K-\sMod$. 
Consider the corresponding faithful exact monoidal $l$-linear functor $\omega_l:\bT_l\to (l\otimes_k K)-\sMod$ sending objects $X$ in $\bT_l$ to $\omega(X)$. The ring $l\otimes_k K$ is either a field or of the form $R:=K[\xi]/(\xi^p)$. In the first case, the theorem is clear.

Otherwise, let $X$ be a nonzero object in  $\bT_l$. 
The $l$-module structure on $X$ induces an $l$-linear morphism $a:l\otimes \unit\to X \otimes X^\ast$, where $X^\ast$ is the dual of $X$ in $\bT$ and $\unit$ the tensor unit of $\bT$. Since $\unit$ is simple in~$\bT$, the kernel of $a$ in $\bT$ has the form $V\otimes\unit$ for some $k$-vector space $V\subset l$. Since $a$ is not zero, it follows that $V\not=l$ and hence cannot admit an $l$-module structure unless $V=0$. Hence $a$ is a monomorphism in $\bT_l$, which is sent by $\omega_l$ to a monomorphism 
$$R\to \omega_l(X)\otimes_R \omega_l(X^\ast)$$ in $R-\sMod$. We conclude that $\omega_l(X)\otimes-$ is a faithful $R$-module for every nonzero object $X$ in $\bT_l$. By applying this observation to all images of actions of $\mZ S_r$ on $X^{\otimes r}$, Lemma~\ref{lem:super} shows that $\omega_l(X)$ is a free $R$-module for all $X$ in $\bT_l$, in particular the endofunctor $X\otimes-$ is exact. By~Proposition~\ref{PropExtSca}, we see that $\bT_l$ is a tensor category.
\end{proof}

\subsection{Semisimple tensor categories}
In the terminology of \cite[Appendix A.3]{AK} the following theorem states that semisimple artinian tensor categories are automatically separable.

\begin{theorem}\label{PropSSE}
	Let $\bT$ be a semisimple artinian tensor category over $k$. Every endomorphism algebra in $\bT$ is separable over $k$. In particular, for every field extension $K/k$ it follows that $K\stackrel{.}{\otimes}_k\bT$ is a semisimple tensor category and equals $\bT_K$. Consequently, $\bT$ is absolute.
\end{theorem}

We start the proofs with the following lemma, which is an immediate extension of \cite[Corollaire~3.6]{Deligne}.
\begin{lemma}\label{LemTr}
	Let $\bT$ be a tensor category equipped with a natural isomorphism $\psi:\id\stackrel{\sim}{\Rightarrow}(-)^{\ast\ast}$ (not necessarily of monoidal functors). For every $Y\in\bT$, the morphism
	$$\alpha_Y:=\Tr^L(\psi_Y\circ -):\,\End(Y)\to k$$
	(see \ref{trace}) sends nilpotent morphisms to zero.
\end{lemma}
\begin{proof}
	Consider first a short exact sequence $X\hookrightarrow Y\tto Z$ and $f\in\End(Y)$ such that $f|_{X}\subset X$. Naturality of $\psi$ implies $\psi_Y\circ f|_X\subset X^{\ast\ast}$. As a direct  consequence of \cite[Proposition~4.7.5]{EGNO} we therefore  find
	$$\alpha_Y(f)\;=\;\alpha_X(f|_X)+\alpha_Z(f|_Z).$$
	By iteration this observation extends to longer chains of subobjects of $Y$. The conclusion now follows from considering a filtration on $Y$ such that the graded morphism associated to the nilpotent morphism $f$ is zero.
\end{proof}

By Corollary~\ref{cor:perf}, it is justified to assume that $k$ is a field of characteristic $p>0$ for the remainder of the section. The following lemma is well-known.

\begin{lemma}\label{LemDiv}
	If $k$ is a separably closed field, then  every finite dimensional division algebra over $k$ is a finite (purely inseparable) field extension of $k$.
\end{lemma}
\begin{proof}
	In order to arrive at a contradiction, we consider a division algebra $D$ over $k$ with non-central element $a\in D$. Since $k(a)/k$ is purely inseparable, we know that $a^q\in k$ for some power $q$ of $p$. It then follows that $\ad_a^q=0$, with $\ad_a=[a,-]:D\to D$.
	Therefore there exists $u\in D$ with $[a,u]=w\not=0$ and $[a,w]=0$. We set $v:=w^{-1}u$, so that $[a,v]=1$. It then follows that for every power $r$ of $p$, we have
	$$(av)^{r}\;=\;a^rv^r - av. $$
	For high enough $r$, two out of three terms will be in $k$. Thus also $av\in k$, which contradicts $[a,v]=1$. Hence $D$ is commutative which concludes the proof.
\end{proof}

\begin{proof}[Proof of Theorem~\ref{PropSSE}]
	Recall that, for an arbitrary extension $K/k$, a $k$-algebra $A$ is separable if and only if the $K$-algebra $K\otimes_kA$ is separable. Moreover, if a $k$-algebra is semisimple and $K/k$ is a separable field extension, then $K\otimes_kA$ is also semisimple.
	Since every pseudo-tensor category with finite dimensional semisimple endomorphism algebras is a semisimple tensor category (see~\cite[Lemma 2]{Jannsen}), we know $K\stackrel{.}{\otimes}_k\bT$ is a semisimple tensor category whenever $K/k$ is separable.
	So, without loss of generality, we may assume that $k$ is separably closed.
	
	Since $\bT$ is semisimple, for every simple object $X$ there is an isomorphism $X\xrightarrow{\sim}X^{\ast\ast}$, see for instance~\cite[Proposition~4.8.1]{EGNO}. We fix such a $g_X:X\xrightarrow{\sim}X^{\ast\ast}$. By Lemma~\ref{LemDiv}, the division algebra $E_X:=\End(X)$ is a finite (purely inseparable) field extension of $k$. As for any purely inseparable extension, we have $\Aut(E_X|k)=\{1\}$. In particular, the automorphism
	$$a\,\mapsto\, g_X^{-1}\circ  a^{\ast\ast}\circ g_X$$
	of $E_X$ is the identity. This means we can collect the isomorphisms $g_X$ and extend them to direct sums in order to create a natural isomorphism $g:\id\stackrel{\sim}{\Rightarrow}(-)^{\ast\ast}$. We consider the morphisms $\alpha_X$ from Lemma~\ref{LemTr}. We claim that, for $a,b\in E_X$, we have
	$$\alpha_X(a)\alpha_X(b)\,=\,\alpha_X(a\circ b)\alpha_X(\id_X).$$
	Indeed, we can write the difference of the two terms as
$$\Tr^L((g_X\circ a\circ b )\otimes g_X\,-\, (g_X\circ a)\otimes (g_X\circ b))
=\alpha_{X\otimes X}(g_{X\otimes X}^{-1}\circ (g_X\otimes g_X)\circ (ab\otimes \id_X -a\otimes b)).$$
	By naturality of $g$, the morphisms $g_{X\otimes X}^{-1}\circ (g_X\otimes g_X)$ and $(ab\otimes \id_X -a\otimes b)$ commute.
	Since $E_X/k$ is purely inseparable, we know $b^{p^n}\in k$ for some $n>0$. Consequently $(ab\otimes \id_X -a\otimes b)$ is nilpotent and hence so is the morphism in the argument of $\alpha_{X\otimes X}$. By Lemma~\ref{LemTr}, the displayed expressions are zero, proving the claim.
	
	Since $\bT$ is semisimple, it is clear that $\Tr^L:E_X\to k$ is nonzero and hence $\alpha_X(\id_X)\not=0$. We thus find a $k$-algebra morphism
	$$f: E_X\to k,\quad a\mapsto \alpha_X(a)/\alpha_X(\id_X),$$
	showing that $E_X=k$. This concludes the proof.
\end{proof}


\section{Deligne's tensor product of tensor categories}\label{SecDel}

\subsection{Connection with abelian envelopes}

\subsubsection{}\label{blablaDTP} Let $\bA$ and $\bB$ be essentially small $k$-linear abelian categories. In \cite[\S 5.1]{Del90}, Deligne defines their tensor product $\bA\boxtimes\bB$ as a $k$-linear abelian category with a universal property. The Deligne tensor product does not always exist, see \cite[Corollary~20]{LF}, but it does when $\bA$ and $\bB$ are artinian, see~\cite{Del90} and~\cite[Proposition 1.11.2]{EGNO}.

If $\bT_1$ and $\bT_2$ are tensor categories for which $\bT_1\boxtimes\bT_2$ exists, 
it follows by \cite[\S 5.16]{Del90} that $\bT_1\boxtimes\bT_2$ admits an essentially unique right exact $k$-linear tensor product such that the functor $\bT_1\stackrel{.}{\otimes}\bT_2\to \bT_1\boxtimes\bT_2$ is monoidal. Here we denote by $\bT_1\stackrel{.}{\otimes}\bT_2\to \bT_1\boxtimes\bT_2$ the idempotent completion of the ordinary tensor product of $k$-linear categories $\bT_1\otimes_k\bT_2$ (see Section~\ref{copr}), to obtain a pseudo-tensor category. When we say that `$\bT_1\boxtimes \bT_2$ is a tensor category', we mean with respect to this monoidal structure. 
If $k$ is perfect and $\bT_1$ and $\bT_2$ are artinian, then $\bT_1\boxtimes\bT_2$ is a tensor category by \cite[Proposition~5.17]{Del90} or~\cite[Proposition 4.6.1]{EGNO}. For arbitrary fields, the question whether $\bT_1\boxtimes\bT_2$ is a tensor category remains open.

\subsubsection{}\label{I1I2I}  Let $\bT_1$ and $\bT_2$ be tensor categories over $k$. We will write $\cV\subset\cU^{ex}(\bT_1\stackrel{.}{\otimes}\bT_2)$ for the class of all morphisms of the form $(u,\id):(U,\unit)\to (\unit,\unit)$ or $(\id,v):(\unit,V)\to (\unit,\unit)$ with $u\in\cU(\bT_1)$ and $v\in\cU(\bT_2)$. It follows from Example~\ref{ExVMon}(2) that $\cV$ is permutable.

\begin{prop}\label{PropDelTP}
Consider $\bT_1$, $\bT_2$ and $\cV$ as in \ref{I1I2I}. The following are equivalent:
\begin{enumerate}
\item The Deligne tensor product $\bT_1\boxtimes\bT_2$ exists (as an abelian category) and is a tensor category.
\item The pseudo-tensor category $\bT_1\stackrel{.}{\otimes}\bT_2$ admits an abelian envelope in which it is strongly generating, and $\cV$ is dense.
\item The category $\Sh(\bT_1\stackrel{.}{\otimes}\bT_2,\cT_\cV)$ is the ind-completion of a tensor category.
\end{enumerate}
Furthermore, the abelian envelope from (2), and the tensor category from (3), is then precisely $\bT_1\boxtimes\bT_2$.
\end{prop}

\begin{proof}
That (2) and (3) are equivalent is a special case of \cite[Corollary~4.4.2]{PreTop}.

We set  $\bD:=\bT_1\stackrel{.}{\otimes}\bT_2$. By Corollary~\ref{CorCan} and \cite[Theorem~3.6.2]{PreTop}, the Kelly tensor product of $\Ind\bT_1$ and $\Ind\bT_2$ as cocomplete $k$-linear categories is precisely $\Sh(\bD,\cT_\cV)$. By \cite[Remark~10 and Theorem~18]{LF}, $\bT_1\boxtimes\bT_2$ exists precisely when the category of compact objects in $\Sh(\bD,\cT_\cV)$ is abelian, and is then equivalent to the latter subcategory. That these equivalences are monoidal follows from uniqueness of the right exact tensor product on $\bT_1\boxtimes\bT_2$, see \ref{blablaDTP}, and the cocontinuous tensor product on $\Sh(\bD,\cT_\cV)$, see~\ref{SecShD}. The equivalence between (1) and (3) follows.
By~\cite[Theorem~4.4.1]{PreTop}, the abelian envelope in (2) is precisely $\bT_1\boxtimes\bT_2$.
 \end{proof}

\begin{remark}
The density of $\cV$, referred to in \ref{PropDelTP}(2), in this case states that for every morphism $(X_1,X_2)\to(\unit,\unit)$ in $\bT_1\otimes \bT_2$, there exist nonzero morphisms $a:Y_1\to\unit $ in $\bT_1$ and $b:Y_2\to\unit$ in $\bT_2$ for which there is a commutative diagram in $\bT_1\otimes \bT_2$
$$\xymatrix{
(X_1,X_2)\ar[rr]&&(\unit,\unit)\\
(Y_1,Y_2)\ar[u]\ar[rru]_{a\otimes b}.
}$$ 
\end{remark}
 


\subsection{Deligne's tensor product for absolute tensor categories}

\begin{theorem}
If $\bT_1$ and $\bT_2$ are absolute artinian tensor categories over $k$, then their Deligne tensor product $\bT_1\boxtimes\bT_2$ is an absolute artininan tensor category.
\end{theorem}


\begin{proof}
	Let $\bar k$ be an algebraic closure of $k$. By assumption, $(\bT_i)_{\bar k}$ is an artinian tensor category for $i=1,2$.
Recall from Proposition~\ref{PropExtSca} that	
$$\Ind(\bT_1)_{\bar{k}}\simeq \Sh(\bar k\stackrel{.}{\otimes}\bT_1,\cT_{(\cU(\bT_1),\id)}),\quad\mbox{and }\quad\Ind(\bT_2)_{\bar{k}}\simeq \Sh(\bar k\stackrel{.}{\otimes}\bT_2,\cT_{(\id,\,\cU(\bT_2))}).$$
	 Now $(\bT_1)_{\bar{k}}\boxtimes _{\bar{k}} (\bT_2)_{\bar{k}}$ exists and is a tensor category over $\bar{k}$ by~\cite[Proposition 4.6.1]{EGNO}. 
	Its ind-completion can be written as 
	$$\Sh(\bar k\stackrel{.}{\otimes}\bT_1\stackrel{.}{\otimes}\bT_2,\cT_{\cV})$$
	by~\cite[Theorem~3.6.2]{PreTop},
	where $\cV\subset\cU^{ex}(\bT_1\stackrel{.}{\otimes}\bT_2)$ is the class of all morphism of the form $(u,\id)$ or $(\id,v)$ with $u\in\cU(\bT_1)$ and $v\in\cU(\bT_2)$.
	Using \cite[Example~3.6.4 and Theorem~3.6.2]{PreTop} again, the above category is equivalent to $(\Sh(\bT_1\stackrel{.}{\otimes}\bT_2,\cT_\cV))_{\bar{k}}$.
	It follows that $\Sh(\bT_1\stackrel{.}{\otimes}\bT_2,\cT_\cV)$ is the ind-completion of a tensor category by~Proposition~\ref{PropInd}(2). 
	Hence Proposition~\ref{PropDelTP} shows that $\bT_1\boxtimes\bT_2$ is a tensor category, and so is its extension of scalars $(\bT_1\boxtimes\bT_2)_{\bar k}$ by~Proposition~\ref{PropExtSca}.
\end{proof}

\begin{corollary}
	The Deligne tensor product of (super-)tannakian categories is an absolute tensor category.
\end{corollary}

\begin{prop}
	The Deligne tensor product of semisimple artinian tensor categories $\bT_1,\,\bT_2$ is a semisimple artinian tensor category and equals $\bT_1\stackrel{.}{\otimes}\bT_2$.
\end{prop}

\begin{proof}
	By~Proposition~\ref{PropSSE}, the endomorphism algebras in $\bT_1$ and $\bT_2$ are separable. The proposition now follows because the tensor product of separable $k$-algebras is a separable $k$-algebra.
\end{proof}



\subsection{Coproducts of symmetric tensor categories}\label{copr}As a variation of Proposition~\ref{PropDelTP}, we will connect the existence of coproducts in the 2-category $\sTens$ with the existence of weak abelian envelopes.

Consider $\bD_1$ and $\bD_2$ in $\sPTens$. 
We denote their tensor product as $k$-linear categories by $\bD_1\otimes\bD_2$. Its objects are pairs $(X,Y)$ with $X\in\bD_1$ and $Y\in\bD_2$. The space of morphisms $(X,Y)\to (X',Y')$
is given by
$$\bD_1(X,X')\otimes_k\bD_2(Y,Y').$$
Note that $\bD_1\otimes\bD_2$ has a canonical symmetric monoidal structure such that $\bD_1\to\bD_1\otimes \bD_2$, given by $X\mapsto (X,\unit)$, and the corresponding functor for $\bD_2$, is symmetric monoidal. The pseudo-abelian envelope $\bD_1\stackrel{.}{\otimes}\bD_2$ of $\bD_1\otimes\bD_2$ is then canonically a symmetric pseudo-tensor category over $k$.

\begin{lemma}\label{PseudoPseudo}
For a symmetric tensor category $\bT$, we have a canonical equivalence
$$\sPTens^{faith}(\bD_1\stackrel{.}{\otimes}\bD_2,\bT)\;\stackrel{\sim}{\to}\;\sPTens^{faith}(\bD_1,\bT)\times\sPTens^{faith}(\bD_2,\bT)$$
\end{lemma}
\begin{proof}
It is well-known that $\bD:=\bD_1\stackrel{.}{\otimes}\bD_2$ is the coproduct of $\bD_1$ and $\bD_2$ in $\sPTens$.
It thus suffices to show that the equivalence
$$\sPTens(\bD,\bT)\;\stackrel{\sim}{\to}\;\sPTens(\bD_1,\bT)\times\sPTens(\bD_2,\bT)$$
exchanges faithful functors on the left with pairs of faithful functors on the right. 
Since the defining functors $\bD_i\to \bD$ are faithful, one direction is clear. Consider faithful linear monoidal functors $F_i:\bD_i\to\bT$. We want to show that
$$\bD_1\otimes\bD_2\to\bT,\quad (X_1,X_2)\mapsto F_1(X_1)\otimes F_2(X_2)$$
is faithful, which means showing that
$$\bD_1(X_1,Y_1)\otimes_k\bD_2(X_2,Y_2)\;\to\; \bT(F_1(X_1)\otimes F_2(X_2),F_1(Y_1)\otimes F_2(Y_2))$$
is injective, for all $X_1,X_2,Y_1,Y_2$.  By letting the map canonically factor through  $$\bT(F_1(X_1),F_1(Y_1))\otimes_k\bT(F_2(X_2),F_2(Y_2)),$$
we write it as a composition of two injective morphisms, by faithfulness of $F_1$ and $F_2$ and Lemma~\ref{LemIndFaith}(1).
This concludes the proof.
\end{proof}


\begin{corollary}\label{CorDelTP}
A diagram $\bT_1\to\bT\leftarrow \bT_2$ in $\sTens$ constitutes a coproduct if and only if the associated (faithful) linear monoidal functor $\bT_1\stackrel{.}{\otimes}\bT_2\to \bT$ is a weak abelian envelope.\end{corollary}
\begin{proof}
 Lemma~\ref{PseudoPseudo} and Theorem~\ref{Thm1} yield, for any tensor category $\bT'$, a commutative diagram
$$
\xymatrix{
\sTens(\bT,\bT')\ar[rr]\ar[rrd]&& \sPTens^{faith}(\bT_1\stackrel{.}{\otimes}\bT_2,\bT')\ar[d]\\
&& \sTens(\bT_1,\bT')\times\sTens(\bT_2,\bT')
}
$$
where the vertical arrow is an equivalence.
This shows that the two universal properties in the corollary are exchangeable.
\end{proof}

\begin{remark}
By Theorem~\ref{MainThmAbEnv} and~Lemma~\ref{PseudoPseudo}, the analogue of Corollary~\ref{CorDelTP} is also true if we replace $\bT_1\stackrel{.}{\otimes}\bT_2$ by $\bD_1\stackrel{.}{\otimes}\bD_2$, for arbitrary strongly generating pseudo-tensor subcategories $\bD_i\subset\bT_i$.
\end{remark}

\section{Further examples and non-examples}\label{SecEx}

\subsection{An example of a weak abelian envelope}\label{ExBEO}


\begin{prop}\label{PropNotAbEnv}
Let $\bT$ be a tensor category over $k$ with projective objects and $\bD$ a pseudo-tensor category over $k$ which is a monoidal, but not necessarily full, subcategory of~$\bT$.
Assume that $\bD$ contains the projective objects, and for every projective object $P$ and every $X\in \bD$,
every morphism $P\to X$ in $\bT$ is contained in $\bD$.
 Then $\bD\to\bT$ is a weak abelian envelope.
\end{prop}
\begin{proof}
This is a special case of Proposition~\ref{PropGF}.
\end{proof}



\begin{example}Let $k$ be a field of characteristic $p>2$.
 We consider the symmetric tensor category $\bT:=\Rep_kC_p$ of representations of the cyclic group $C_p=\mZ/p$ of order $p$. 
 Consider the additive monoidal subcategory $\bD$ of $\bT$ which contains the indecomposable modules (labelled by their dimension) $\unit=M_1$, $M_{p-1}$ and $kC_p\simeq M_p$, and which is full in $\bT$ except for
 $$\bD(M_1,M_{p-1})=0=\bD(M_{p-1},M_1).$$ In \cite[Example~2.44(3)]{BEO}, it is demonstrated that 
 $\bD$ is a monoidal subcategory of $\bT$. Clearly the conditions in Proposition~\ref{PropNotAbEnv} are satisfied. Hence $\bD$ admits a weak abelian envelope (which is $\bT$), but no abelian envelope.
\end{example}

\subsection{Not every tensor category is self-splitting}\label{RepGa} 

\subsubsection{} Let $k$ be a field of characteristic zero and consider the additive group $\mG_a$ of $k$. The symmetric tensor category $\Rep\mG_a$ is unipotent and the indecomposable representations (including the zero representation) can be labelled as $M_i$, $i\in\mN$, by their dimension. 

\begin{lemma}\label{LemGa} Consider the (pseudo-)tensor category $\bD=\Rep\mG_a$.
 For a monomorphism $u:\unit\simeq M_1\hookrightarrow M_2$, there is no $X\in \Rep\mG_a$ for which $u\otimes X$ is split.
\end{lemma}
\begin{proof}
By \cite[Lemma~2.1]{BEO} it suffices to prove that $M_i$ does not split $u$, for all $i\in\mZ_{>0}$. However, we have
$$u\otimes M_i\,:\, M_i\hookrightarrow M_{i-1}\oplus M_{i+1},$$
which cannot be split.
\end{proof}

\begin{remark}
We have already observed that $\bD=\Rep\mG_a$ is, as any tensor category, its own abelian envelope. However, this fact does not fit into the theories of \cite{BEO, AbEnv1}. Indeed, $\bD$ is not `separated' in the terminology of \cite{BEO}, by Lemma~\ref{LemGa}, so \cite[Theorem~1.2]{BEO} does not apply. It also follows easily that the monoidal Grothendieck category constructed in \cite[Theorem~A]{AbEnv1} is simply $\PSh\bD$, so in particular not $\Ind\bD$. This demonstrates the need for the theory of abelian envelopes as developed in the current paper. 
\end{remark}

\subsection{Tilting modules over a reductive group}

Let $G$ be a reductive group over an algebraically closed field $k$ of characteristic $p>0$ and let
$\Tilt(G)\subset \Rep(G)$ be the category of tilting modules, see \cite[Chapter~E]{Jantzen}. By \cite[Remark~3.3.2]{CEH} and Theorem~\ref{MainThmAbEnv}, we find the following.
\begin{prop}\label{TiltEnv}
The tensor category $\Rep(G)$ is the abelian envelope of $\Tilt(G)$.
\end{prop}
Note that the above is just the generalisation of \cite[Theorem~3.3.1]{CEH} which ignores braidings.

For any $r\in \mN$
let $\St_r=T((p^r-1)\rho)$ be the Steinberg module (we henceforth exclude $G$ for which the latter does not exist, which can only occur if $p=2$). 
Let $\bC$ be
a pseudo-tensor category and let $F: \Tilt(G)\to \bC$ be a linear monoidal functor.

\begin{prop}\label{SteiTest} If $F$ is not faithful, there is $r\in\mN$ such that
$F(\St_r)=0$.
\end{prop}

\begin{proof} Let $f: X\to Y$ be a morphism in $\Tilt(G)$ such that $F(f)=0$. Then for sufficiently large
$s$ we have that $f\otimes \St_s$ splits, see~\cite[Remark 3.3.2]{CEH}, and its image is some tilting module. Thus $F(T(\lambda))=0$
for some $\lambda$. Hence $F(\St_r)=0$ if $(p^r-1)\rho -\lambda$ is dominant.
\end{proof}

The combination of Propositions~\ref{TiltEnv} and~\ref{SteiTest} now yields.
\begin{corollary} Assume $\bC$ is a tensor category and $F(\St_r)\ne 0$ for any $r$. Then
$F$ uniquely extends to a tensor functor $\tilde F: \Rep(G)\to \bC$.
\end{corollary}

\begin{example} Let $G=SL_{2n+1}$ and let $q$ be a prime satisfying $q\equiv 1\pmod{p}$. 
Given a positive integer $k$, we write $[k]_q:=\frac{q^k-1}{q-1}\in\mN$. Setting $[1]_q!:=1$, we define $[k+1]_q!:=[k+1]_q[k]_q!$ recursively. For $i\leq k$, the $q$-binomial coefficient is then defined as $\binom{k}{i}_q:=\frac{[k]_q!}{[i]_q![k-i]_q!}$. 

Given some combinatorial data,
in \cite{Jo} the author constructs a linear monoidal functor $F: \Tilt(G)\to \Vecc$ sending the fundamental
representation $T(\omega_i)$ to a vector space of dimension $\binom{2n+1}{i}_q$. 
Recall that
the split Grothendieck ring of $\Tilt(G)$ identifies with the ring of symmetric polynomials in
variables $x_1, x_2, \ldots, x_{2n+1}$ modulo the relation $x_1x_2\cdots x_{2n+1}=1$, see~\cite{Do}. 
Let $e_1, e_2, \ldots ,e_{2n+1}$ be the elementary symmetric functions in the variables
$x_1, x_2, \ldots, x_{2n+1}$ and let $z_1, z_2, \ldots, z_{2n+1}\in \mC$ be complex numbers satisfying
$$e_i(z_1, z_2, \ldots, z_{2n+1})=\binom{2n+1}{i}_q,\quad\mbox{for $1\le i\le 2n+1$}.$$
Then $\dim F(\St_r)=S_\lambda(z_1,\ldots, z_{2n+1})$ where $S_\lambda$ is the Schur function
computing the character of $\St_r$. 

Observe that the numbers $z_1, z_2, \ldots, z_{2n+1}$ are the roots of the polynomial
$$\sum_{i=0}^{2n+1}\binom{2n+1}{i}_q(-x)^i.$$
A result of Hutchinson \cite{Hut} states that all the roots of a polynomial $\sum_{i=0}^k
a_ix^i$ with positive real coefficients $a_i$ satisfying $a_i^2>4a_{i-1}a_{i+1}, i=1,\ldots, k-1$ are real and negative.
Hence the easily verified inequalities 
$$\binom{2n+1}{r}_q^2>4\binom{2n+1}{r-1}_q\binom{2n+1}{r+1}_q,\; r=1,2,\ldots, 2n$$
imply that the numbers $z_1, z_2, \ldots, z_{2n+1}$ are real and positive. Since $S_\lambda$ is a positive combination of monomials we conclude
that $\dim_k F(\St_r)>0$. Thus $F(\St_r)\ne 0$ and the functor $F$ extends uniquely to a tensor functor $\widetilde F: \Rep(SL_{2n+1})\to \Vecc$.

\end{example}

\subsection{Tilting modules over quantum groups}

Let $\mathfrak{g}$ be a simple complex Lie algebra and suppose $k$ has characteristic $0$. 
Consider Lusztig's quantum group $U_q(\mathfrak{g})$, where $q$ is a root of unity in~$k$. We let $\Tilt(U_q(\mathfrak{g}))\subset \Rep(U_q(\mathfrak{g}))$ be the category of tilting modules, see~\cite{An1}, 
and write $\St$ for the Steinberg module as defined in~\cite[\S 4.4]{An2}. 
By~\cite[Corollary 4.7]{An2}, every object in $\Rep(U_q(\mathfrak{g}))$ embeds into a tilting module, so Theorem~\ref{MainThmAbEnv} implies the following.
\begin{prop}\label{TiltQuant}
	The tensor category $\Rep(U_q(\mathfrak{g}))$ is the abelian envelope of $\Tilt(U_q(\mathfrak{g}))$.
\end{prop}


Let $\bC$ be a pseudo-tensor category and let $F: \Tilt(U_q(\mathfrak{g}))\to \bC$ be a linear monoidal functor.

\begin{prop}\label{SteiTestQuant} If $F$ is not faithful, then $F(\St)=0$.
\end{prop}

\begin{proof} Let $f: X\to Y$ be a morphism in $\Tilt(U_q(\mathfrak{g}))$ such that $F(f)=0$. 
By~\cite[Corollary 4.7]{An2} 
it follows that $f\otimes \St$ splits and $F(M)=0$ for its image $M$. The monomorphism $\St\hookrightarrow M\otimes M^{\ast} \otimes \St$ induced from $\co_M$ is split, showing that $F(\St)=0$.
\end{proof}

\begin{corollary} Assume $\bC$ is a tensor category and $F(\St)\ne 0$. Then
	$F$ uniquely extends to a tensor functor $\tilde F: \Rep(U_q(\mathfrak{g}))\to \bC$.
\end{corollary}

\subsection{Non-degenerate pseudo-tensor categories}

A pseudo-tensor category $\bD$ is non-degenerate if one of the following equivalent conditions holds:
\begin{enumerate}[label=(\alph*)]
\item For every nonzero $f:X\to Y$ in $\bD$, there exists $g:Y\to X^{\ast\ast}$ for which $\Tr^L(g\circ f)\not=0$.
\item Every morphism $U\to \unit$ is split.
\item Every morphism $\unit\to U$ is split.
\end{enumerate}
\begin{lemma}
If $\bD$ is a non-degenerate pseudo-tensor category, then $\cU=\cU^{ex}$ and $\bD$ is strongly generating in a tensor category if and only if $\bD$ is a tensor category itself.
\end{lemma}
\begin{proof}
The first statement follows immediately from version (b) of the definition (and the fact that $\unit$ is indecomposable). Now assume that $\bD$ strongly generates a tensor category $\bT$. It follows that every epimorphism $X\tto\unit$ in $\bT$ splits. Since
$$\Ext^1(A,B)\;\simeq\;\Ext^1(\unit, B\otimes A^\ast),$$
for all $A,B\in \bT$, this means that every morphism in $\bT$ is split. By assumption, every object in $\bT$ is a quotient of an object in $\bD$. The above thus means that every object in $\bT$ is a direct summand of one in $\bD$, whence $\bD\simeq\bT$.
\end{proof}


\begin{lemma}\label{LemND}
	Let $\bD$ be a non-degenerate pseudo-tensor category with finite dimensional morphism spaces, and suppose $\bT$ is a tensor category equipped with a natural isomorphism $\psi:\id\stackrel{\sim}{\Rightarrow}(-)^{\ast\ast}$. If there exists a fully faithful functor from $\bD$ to $\bT$, then $\bD$ is a (semisimple) tensor category. 
\end{lemma}
\begin{proof}
	Suppose $\bD$ embeds fully faithfully in $\bT$. Since $\bD$ has finite dimensional morphism spaces, it suffices to show that $\End(X)$ is semisimple for every object $X$ in $\bD$. Let $f$ be in the Jacobson radical of $\End(X)$ and suppose $f\not= 0$. By non-degeneracy, we can find a morphism $h:X\to X$ for which $\Tr^L(\psi_Y\circ h\circ f)\not=0$ in $\bT$. This contradicts~Lemma~\ref{LemTr}. 
\end{proof}

\begin{corollary}
	A triangulated pseudo-tensor category $\bD$ with a fully faithful embedding into a pivotal artinian tensor category is a (semisimple) tensor category.
\end{corollary}
\begin{proof}
By Example~\ref{ExTria}, $\bD$ is non-degenerate, so we can apply Lemma~\ref{LemND}.
\end{proof}

\appendix

\subsection*{Acknowledgement}
The work of K.C. and B.P. was supported by grant DP180102563 of the Australian Research Council. The work of P.E. was partially supported by the NSF grant DMS-1916120. The work of V. O. was partially supported by the NSF grant DMS-1702251 and by the Russian Academic Excellence Project `5-100’.


\begin{thebibliography}
	{EGNO}








\bibitem[An1]{An1} H.~H.~Andersen: Tensor products of quantized tilting modules. Comm. Math. Phys. 149 (1992), 149--159.

\bibitem[An2]{An2} H.~H.~Andersen: The strong linkage principle for quantum groups at roots of 1.
J. Algebra, 260 (1) (2003), pp. 2--15.

\bibitem[AK]{AK}
	Y.~Andr\'e, B.~Kahn:
Nilpotence, radicaux et structures mono\"idales.
With an appendix by Peter O'Sullivan. 
Rend. Sem. Mat. Univ. Padova 108 (2002), 107--291.


\bibitem[BE]{BE} D.~Benson, P.~Etingof: Symmetric tensor categories in characteristic 2. Adv. Math. 351 (2019), 967--999.


\bibitem[BEO]{BEO} D.~Benson, P.~Etingof, V.~Ostrik: New incompressible symmetric tensor categories in positive characteristic.  arXiv:2003.10499.



\bibitem[BHM]{BBHM} A.~Białynicki-Birula, G.~Hochschild, G.D.~Mostow:
Extensions of representations of algebraic linear groups. 
Amer. J. Math. 85 (1963), 131--144. 






\bibitem[CO]{CO}
J.~Comes, V.~Ostrik:
On Deligne's category $\underline{Rep}^{ab}(S_d)$.
Algebra Number Theory 8 (2014), no. 2, 473--496. 

\bibitem[Co1]{AbEnv1} K. Coulembier: Monoidal abelian envelopes. Compos. Math. 157 (2021), no. 7, 1584--1609.

\bibitem[Co2]{PreTop} K. Coulembier: Additive Grothendieck pretopologies and presentations of tensor categories. arXiv:2011.02137.




\bibitem[CEH]{CEH} K. Coulembier, I. Entova-Aizenbud, T. Heidersdorf: Monoidal abelian envelopes and a conjecture of Benson - Etingof.  arXiv:1911.04303.

\bibitem[CSV]{CSV} K. Coulembier, R. Street and M. van den Bergh: Freely adjoining monoidal duals. Math. Structures Comput. Sci. 31 (2021), no. 7, 748--768.




	
	\bibitem[DM]{DM} P.~Deligne, J.S.~Milne: Tannakian Categories. In
Hodge cycles, motives, and Shimura varieties. 
Lecture Notes in Mathematics, 900. Springer-Verlag, Berlin-New York, 1982, pp. 101-228.

	
	\bibitem[De1]{Del90} P.~Deligne: Cat\'egories tannakiennes. The Grothendieck Festschrift, Vol. II, 111--195, Progr. Math., 87, Birkh\"auser Boston, Boston, MA, 1990. 
	
	
	
\bibitem[De2]{Deligne}
P. Deligne:
La cat\'egorie des repr\'esentations du groupe sym\'etrique $S_t$, lorsque $t$ n'est pas un entier naturel.  Algebraic groups and homogeneous spaces, 209--273, 
{Tata Inst. Fund. Res. Stud. Math.}, Mumbai, 2007. 

\bibitem[De3]{Del14}
P.~Deligne: Semi-simplicit\'e de produits tensoriels en caract\'eristique p. Invent. Math. 197 (2014), no. 3, 587--611.


\bibitem[DG]{DG} M.~Demazure, P.~Gabriel: Groupes algébriques. Tome I: Géométrie algébrique, généralités, groupes commutatifs. Masson \& Cie, Éditeurs, Paris; North-Holland Publishing Co., Amsterdam, 1970. 

\bibitem[Do]{Do} S.~Donkin: On tilting modules for algebraic groups. Math. Z. 212 (1993), no. 1, 39--60.

		

 \bibitem[EO]{EO} P.~Etingof, V.~Ostrik: On the Frobenius functor for symmetric tensor categories in positive characteristic. arXiv:1912.12947.

\bibitem[EGNO]{EGNO}P.~Etingof, S.~Gelaki, D.~Nikshych, V.~Ostrik:
Tensor categories. 
Mathematical Surveys and Monographs, 205. American Mathematical Society, Providence, RI, 2015. 



\bibitem[EHS]{EHS}
I.~Entova-Aizenbud, V.~Hinich, V.~Serganova:
Deligne categories and the limit of categories~$Rep(GL(m|n))$.
Int. Math. Res. Not. IMRN 2020, no. 15, 4602--4666.

  
  

\bibitem[Gr]{Gr} F.D.~Grosshans: Algebraic homogeneous spaces and invariant theory. Lecture Notes in Mathematics, 1673. Springer-Verlag, Berlin, 1997.



\bibitem[Hu]{Hut} J.I. Hutchinson: On a Remarkable Class of Entire Functions, Transactions of the American Mathematical Society, Vol. 25, No. 3 (Jul., 1923), pp.325-332.





\bibitem[Ja1]{Jannsen} U.~Jannsen: Motives, numerical equivalence, and semi-simplicity. Inventiones Mathematicae 107 (1992), 447-452.

\bibitem[Ja2]{Jantzen}
J.C.~Jantzen:
Representations of algebraic groups. 
Second edition. Mathematical Surveys and Monographs, 107. American Mathematical Society, Providence, RI, 2003.

\bibitem[Jo]{Jo} C.~Jones: Triangle presentations and titling modules for $SL_{2k+1}$,
arxiv:2005.07172.


\bibitem[LF]{LF} F.I.~L\'opez Franco:
Tensor products of finitely cocomplete and abelian categories.
J. Algebra 396 (2013), 207--219. 

\bibitem[Ne]{Ne} A.~Neeman: Triangulated categories. Annals of Mathematics Studies, 148.
Princeton University Press, Princeton, NJ, 2001.

\bibitem[Ro]{Ro} J.E.~Roos: Locally Noetherian categories and generalized strictly linearly compact rings. Applications. 
Category Theory, Homology Theory and their Applications, II (Battelle Institute Conference, Seattle, Wash., 1968, Vol. Two), pp. 197–277. Springer, Berlin, 1969.








 
 	\end{thebibliography}
\end{document}